\documentclass[10pt]{article}
\usepackage{amsmath}
\usepackage{amsfonts}
\usepackage{amssymb}
\usepackage{amsthm}
\usepackage{url}
\usepackage[all]{xy}
\usepackage{dsfont}
\usepackage{graphicx}
\usepackage{caption}
\usepackage{subcaption}
\usepackage{comment} 
\usepackage{stmaryrd}
\usepackage{hyperref}
\usepackage{todonotes}
\usepackage{color}
\usepackage{enumerate,tikz-cd}
\usepackage{fixltx2e}
\usepackage{MnSymbol}
\usepackage{comment}
\allowdisplaybreaks

\numberwithin{equation}{section}

\newtheorem{proposition}{Proposition}[section]
\newtheorem{lemma}[proposition]{Lemma}
\newtheorem{theorem}[proposition]{Theorem}
\newtheorem{corollary}[proposition]{Corollary}

\theoremstyle{definition}
\newtheorem{remark}[proposition]{Remark}
\newtheorem{definition}[proposition]{Definition}

\DeclareMathOperator{\diag}{diag}

\renewcommand{\epsilon}{\varepsilon}

\newcommand{\mfl}{\mathfrak{l}}

\renewcommand{\phi}{\varphi}

\newcommand\an{^{\mathrm{an}}}
\newcommand\mi{^{-1}}
\newcommand\vol{\mathrm{vol}}

\newcommand\bba{\mathbb{A}}

\newcommand\bbq{\mathbb{Q}}
\newcommand\bbr{\mathbb{R}}

\newcommand\cC{\mathcal{C}}

\newcommand\cE{\mathcal{E}}

\newcommand\cN{\mathcal{N}}

\newcommand\cR{\mathcal{R}}

\newcommand{\norm}[1][\cdot]{\left\|#1\right\|}

\newcommand\triv{{\mathrm{triv}}}

\newcommand\field{{\mathrm{k}}}

\renewcommand\diag{{\mathrm{diag}}}

\newcommand\zb{{0,\bullet}}
\newcommand\bt{{\bullet,t}}

\title{Filtrations and asymptotic geometry of non-Archimedean norms on section rings.}
\author{Rémi Reboulet}

\begin{document}

\maketitle

\begin{abstract}
   	This article is concerned with the metric study of a construction of Gérardin of the action of the boundary at infinity of the space of norms on a non-Archimedean vector space, and its generalisation to graded algebras. Namely, given $(X,L)$ a polarised variety over an arbitrary non-Archimedean field, we show that there is a jointly $d_1$-contracting action of the space of filtrations of the section ring $R(X,L)$ on the space of graded norms on $R(X,L)$. This naturally yields non-Archimedean geodesic rays and infinite-dimensional flats in this setting, generalising previous work of the author and Witt Nyström. It is further shown that relative limit measures converge along geodesic rays, providing a result on the $d_p$-radial geometry of graded norms, analogous to a recent result of Finski in the Archimedean case.
\end{abstract}

\tableofcontents

\section*{Introduction.}

\paragraph{The Goldman--Iwahori space and the Gérardin action.} Motivated by non-Archimedean quadratic forms, Goldman and Iwahori \cite{goldmaniwahori} (see also Mahler \cite{mahler} for a different perspective) investigate the space of norms on a finite-dimensional $p$-adic vector space, proving deeply influential results on its topology and metric structure, showing how close and how far it is from the better understood space of Hermitian matrices. 

Through the years, the more general space $\cN(V)$ of non-Archimedean norms on a finite-dimensional vector space over general non-Archimedean fields has seen much interest, from the point of view of algebraic groups and Bruhat--Tits buildings \cite{iwahorimatsumoto:ihes} \cite{schneiderstuhler:inventiones}, the Langlands correspondence \cite{drinfeld,fargues:progressinmath}, and more recently non-Archimedean pluripotential theory \cite{boueri,reb:1,reb:2} and its connections with Kähler geometry and the Yau--Tian--Donaldson conjecture \cite{bbj,bj:kstab1,reb:3,abbanzhuang}.

Much as its Hermitian analogue, $\cN(V)$ has been identified as a negatively curved metric space, and the purpose of this article is to follow the natural direction of understanding its asymptotic metric geometry. The main existing precise work in this direction, beyond general results on Euclidean buildings, is due to Gérardin \cite{book:gerardin}, in which it is shown that the space of directions at infinity in $\cN(V)$ are loosely identified with \textit{filtrations} of $V$. He defines a \textit{geodesic action} of the space of filtrations and studies some of its properties.

In this work, we make precise this boundary at infinity picture, and extend it to the infinite-dimensional setting of submultiplicative norms on graded algebras.

First begin with $V$ a finite-dimensional $\field$-vector space, with $(\field,|\cdot|)$ non-Archimedean. We consider two spaces simultaneously: the space $\cN(V)$ of \textit{norms} on $V$, that is: functions $\norm:V\to\bbr_{\geq 0}$ that are definite, $|\cdot|$-homogeneous and satisfy the ultrametric inequality; and the space $\cN_0(V)$ of \textit{trivially-valued norms} on $V$, which are $|\cdot|_0$-homogeneous, where $|\cdot|_0$ is the \textit{trivial} absolute value equal to $1$ on $\field^\times$. Elements in $\cN_0(V)$ are in canonical bijection with filtrations of $V$, but the norm perspective will make many constructions clearer.

Both spaces may be endowed with various metric structures, beginning with the $d_\infty$-distance pioneered in \cite{goldmaniwahori}:
$$d_\infty(\norm,\norm')=\sup_{v\in V}\left|\log\frac{\norm[v]'}{\norm[v]}\right|,$$
but also $d_p$ distances for $1\leq p<\infty$, which roughly correspond with the normalised absolute $p$-th moments of their relative spectral measure $\sigma(\norm,\norm')$ -- a non-Archimedaen analogue of the sum of Dirac masses at the logarithmic eigenvalues of the transition matrix between two Hermitian norms. The case $p=2$ is the one used in the theory of Bruhat--Tits buildings, and $(\cN(V),d_2)$ is CAT(0). The case $p=1$ is closely related to determinants of norms, which traces its roots back to Mahler \cite{mahler}. The general $p$ case was introduced in \cite{boueri}, inspired by work of Darvas \cite{dar:mabuchi} and Darvas--Lu--Rubinstein \cite{dar:quanti} in the Hermitian setting, itself going back to Bhatia \cite{bhatia}.

The space $(\cN(V),d_p)$ is geodesic, as was seen already by \cite{goldmaniwahori} over a local field, and systematically studied in \cite{reb:2}, and negatively curved, justifying the interest in the study of its asymptotic geometry. Gérardin \cite{book:gerardin} introduces the so-called \textit{geodesic action} of $\cN_0(V)$ on $\cN(V)$ (over a local field): given a trivially-valued norm (equivalently a filtration) $\norm_0\in \cN_0(V)$, the envelope
$$\norm_0.\norm[v]:=\sup\{\norm[v]',\,\norm'\leq\norm\cdot\norm'\text{ on V}\}$$
admits an upper bound in $\cN(V)$, and using the scaling action $t.\norm_0$ of $\cN_0(V)$ turns out to yield geodesic rays
$$t\mapsto (t.\norm_0).\norm\in \cN(V).$$
We study the metric properties of this geodesic action, and in particular prove that $(\cN_0(V),d_p)$ is the boundary at infinity of $(\cN(V),d_p)$, as is the case for Hermitian matrices (see \cite[(1-2)]{bou:icm}).

\paragraph{The graded setting.} In the Archimedean case, given a complex projective manifold $X$ endowed with an ample line bundle $L$, filtrations of the section ring of $L$ play a major role in the study of K-stability \cite{bj:kstab1,bj:kstab2,sze:filtr} and are realised in a certain way as the boundary at infinity of the space of Kähler metrics (or its completions) on $L$, see \cite{finski:dp}. Kähler and more general plurisubharmonic metrics on $L$ are themselves realised by sequences of Hermitian norms on the section ring of $L$ \cite{bk:tian,bk:zelditch,berbou,dar:quanti,finski:metricstructure}. 

Analogously, pick $X$ a smooth projective variety over $\field$, and $L$ an ample line bundle on $X$. The set $\cN(R)$ of graded (i.e.\ bounded and submultiplicative) non-Archimedean norms on the section ring $R:=R(X,L)$ of $L$ may be interpreted as a set of non-Archimedean plurisubharmonic metrics on $L$ \cite{boueri}, and its trivially-valued analogue $\cN_0(R)$ should, from the finite-dimensional picture, realise its boundary at infinity.

We first describe its metrisation. By an equidistribution result of Chen--Maclean \cite{cmac}, given two graded norms $\norm_\bullet,\norm'_\bullet$ on the section ring of $L$, the relative spectral measures $m\mi_*\sigma(\norm_m,\norm'_m)$ converge weakly to a compactly supported measure $\sigma(\norm_\bullet,\norm'_\bullet)$ often called their \textit{relative limit measure}. In particular, one can define $d_p$-(pseudo)metrics on $\cN(R)$ by
$$d_p(\norm_\bullet,\norm'_\bullet)^p=\int |x|^p d\sigma(\norm_\bullet,\norm'_\bullet)=\lim_{m\to\infty}m\mi d_p(\norm_m,\norm'_m).$$
The induced metric quotient is further independent of $p$. In practice, our results will pass to the quotient (and the metric completion), but for clarity of presentation we will omit taking the quotient in the expression of our results in this introduction. Building on our results concerning Gérardin's action in the finite-dimensional setting, we obtain in particular:

\bigskip\noindent \textbf{Theorem A.} Let $(X,L)$ be a polarised smooth variety over an arbitrary non-Archimedean field $(\field,|\cdot|)$, and let $R:=R(X,L)$ be its section ring. There is an action $\cN_0(R)\times \cN(R)\to \cN(R)$ which:
\begin{enumerate}
    \item is $1$-Lipschitz in both variables with respect to the $d_1$ metric;
    \item transforms the scaling action on $\cN_0(R)$ into geodesic rays.
\end{enumerate}

Building on Theorem A and our previous results on non-Archimedean geodesics in \cite{reb:2}, and also on \cite{goldmaniwahori} and work of Parreau \cite{parreau:immeubles}, we study geodesic rays in $\cN(R)$, providing a non-Archimedean version of the result of Phong--Sturm \cite{ps:geodesics}. Further, given a geodesic ray in $\cN(R)$ we define its \textit{trivially-valued limit} in $\cN_0(R)$, which in our context is an analogue of the non-Archimedean limit operation of Kähler geodesic rays as in \cite{bbj} (see also \cite{reb:3}). Our second main theorem relates the radial geometry of geodesic rays in $\cN(R)$ with the geometry of $\cN_0(R)$, and effectively states that $(\cN_0(R),d_p)$ is the boundary at infinity of $(\cN(R),d_p)$, akin to the limit results in \cite{reb:3,finski:dp} (see also \cite{bdl:kenergy}) in the Kähler case. Our result also has a much stronger formulation in terms of relative limit measures, the non-Archimedean version of \cite[Theorem 2.2]{finski:dp}:

\bigskip\noindent \textbf{Theorem B.} Let $t\mapsto \norm_{\bullet,t},\norm'_{\bullet,t}$ be geodesic rays in $\cN(R)$. Let $\norm_\zb$, $\norm'_\zb$ be their trivially-valued limits. Then
$$\lim_{t\to\infty}t\mi_* \sigma(\norm_{\bullet,t},\norm'_{\bullet,t})=\sigma(\norm_\zb,\norm'_\zb).$$
In particular, for all $p\in [1,\infty)$,
$$\lim_{t\to\infty}t\mi d_p(\norm_{\bullet,t},\norm'_{\bullet,t})=d_p(\norm_\zb,\norm'_\zb)$$

\bigskip This induces a precise radial isometry result, see Corollary \ref{coro:isometry}.

Finally, using Theorem A and modifying a construction of \cite{reb:wn} we then find:

\textbf{Theorem C.} Any element in $\cN(R)$ lies at the apex of infinitely many infinite-dimensional $d_p$-metric flat cones, for $p\in [1,\infty)$.

\bigskip A more precise result is stated in Theorem \ref{thm:b}. By a metric flat, we mean an isometric embedding from a convex subset of a (strictly convex) normed vector space -- geodesics and geodesic rays being examples of one-dimensional metric flats. This result should be seen from the point of view of Bruhat--Tits buildings, which the finite-dimensional $\cN(V)$ can be realised as: any norm belongs to an \textit{apartment}, isometric to $\bbr^{\dim_\field V}$. This result therefore provides a first glimpse into the infinite-dimensional building structure of $\cN(R)$.

\paragraph{Strategy of proof.}  Theorem A follows by quantisation from its finite-dimensional analogue, which is already completely new and requires a deep study of the $d_1$ and $d_\infty$-metric properties of the Gérardin action on the Goldman--Iwahori space $\cN(V)$. The two fundamental results are Theorems \ref{thm:d1envelope} and \ref{thm:dinftyenvelope}. The former requires determinant arguments building on \cite{boueri,reb:2} while the latter, especially in the non-diagonalisable case, requires involved results on almost-orthogonalisations and insights from \cite{goldmaniwahori}.

While Theorem \ref{thm:convergencemeasurefindim}, which is the finite-dimensional analogue of Theorem B, is almost trivial once the right formalism is developed, its proof does not work in the graded setting. The proof of Theorem B is much more difficult, relying on a technical result of \cite{cmac} and a volume argument inspired by \cite{bdl:kenergy,reb:3}, adapted to the setting of non-Archimedean graded norms.

Theorem C uses the arguments of \cite{reb:wn}, with a slight modification of its main construction.

\paragraph{Thoughts.} The $d_p$-analogue of Theorem \ref{thm:d1envelope} and therefore of Theorem A, though not required for our main results, remains unclear. While one can show a $d_p$-type uniform convexity result, obtaining an $L^p$-analogue of the CAT(0) condition using the framework of Euclidean buildings, it turns out as was pointed out by Tamás Darvas to the author that it is an open problem in metric spaces in general whether $L^p$ CAT(0) conditions imply Busemann convexity.

As mentioned in \cite{reb:wn} already (but also \cite{donaldson:symmetric,codogni:tits}), it is desirable to either upgrade the result from Theorem C to a more complete building-style picture, or to understand precisely why such an upgrade is impossible. More precisely, buildings have the fundamental property that any two elements share an apartment, i.e.\ belong to a same flat subspace. This is clearly not true of our metric flat cones, suggesting the picture to be more subtle than what is given by Theorem C.

In a recent article, Darvas--McCleerey \cite{dar:lines} study geodesic \textit{lines} in the space of Kähler metrics, and show existence of many such lines in full generality, disproving a conjecture of Berndtsson. This is a surprising result: from the perspective of Euclidean buildings, one would expect existence of such nontrivial lines only in the case where $(X,L)$ admits a $\mathbb{G}_m$ action, which is indeed the algebraic version of the aforementioned conjecture. It is clear that the naive quantisation construction of geodesic lines in $\cN(R)$ fails, suggesting a new approach is required. The author therefore aims to pursue this direction in later work.

Finally, in Section \ref{sect:2na}, we very briefly touch upon the non-Archimedean pluripotential theory interpretation of our results. This is relegated to such a short section due to two problems: the need for continuity of envelopes, which is partially bypassed by arguments from \cite{bj:trivval,bj:synthetic}, but more importantly the lack of characterisation of non-Archimedean metrics in $\cE^p$ in terms of pluripotential theory, for $p>1$, in the spirit of \cite{dar:mabuchi}.

\paragraph{Organisation of the paper.} The article is clearly divided into a section on the finite-dimensional setting of norms on a finite-dimensional vector space, and on the infinite-dimensional setting of submultiplicative norms on the section ring of a line bundle.

We mostly recall known facts from \cite{goldmaniwahori,book:gerardin,boueri} in Sections \ref{sect:11}-\ref{sect:13}, providing proofs for improvements of results therein and various likely folklore results that we could not find written explicitly in the literature. In Sections \ref{sect:14} and \ref{sect:15}, we prove two of the three main results of Section 1 in Theorems \ref{thm:d1envelope} and \ref{thm:dinftyenvelope}. Section \ref{sect:16} recalls facts on geodesics from \cite{reb:2}, but requires a few technical results due to following the conventions of \cite{goldmaniwahori}.  In Section \ref{sect:17}, we study geodesic rays, and use work of Parreau \cite{parreau:immeubles}.

Sections \ref{sect:21}-\ref{sect:24} mostly consist in setting up the graded picture and showing that many of our finite-dimensional results extend by quantisation. Theorem A is thus proven in passing in Section \ref{sect:24}. Section \ref{sect:25} contains the proof of Theorem B. Section \ref{sect:26} proves Theorem C, building on \cite{reb:wn}. Section \ref{sect:2na} discusses the non-Archimedean pluripotential theory picture.

\paragraph{Acknowledgements.} The author thanks Sébastien Boucksom, Tamáas Darvas, and Siarhei Finski for discussions related to the matter of this article.

\section*{Notation.}

By \textit{non-Archimedean field} we mean a field $\field$ endowed with an ultrametric absolute value $|\cdot|$ such that $\field$ is Cauchy-complete with respect to the topology induced by $|\cdot|$. We also denote by $|\cdot|_0$ the \textit{trivial} absolute value on $\field$, which is equal to $1$ on $\field-\{0\}$. Objects (such as norms) related to the trivial absolute value will usually be denoted with a ${}_0$ index. This occasionally creates confusion when considering the time zero evaluation of a geodesic (ray), which we try to mitigate as much as possible.

\section{Asymptotic geometry of the space of norms.}

Throughout, we fix a non-Archimedean field $(\field,|\cdot|)$, and $V$ a $\field$-vector space with $n:=\dim_\field V<\infty$. 

\subsection{Preliminaries: the Goldman--Iwahori space.}\label{sect:11}

We define a \textit{non-Archimedean norm} on $V$ to be a positive definite, $|\cdot|$-homogeneous function $\norm:V\to\bbr_{\geq 0}$ satisfying the ultrametric inequality:
$$\norm[v+v']\leq\max(\norm[v],\norm[v']).$$
We denote by $\cN(V)$ the set of all norms on $V$. We also denote by $\cN_0(V)$ the set of non-Archimedean norms on $V$ with respect to the trivial absolute value $|\cdot|_0$. Note that $\cN_0(V)$ contains the \textit{trivial norm} $\norm_\triv$, equal to $1$ on $V-\{0\}$.

There is a translating action of $\bbr$ on $\cN(V)$, given by $e^{c}\norm$. There is further a \textit{scaling action} of $\bbr_{>0}$ on $\cN_0(V)$, given by
\begin{equation}
    t.\norm_0:=(\norm_0)^t.
\end{equation}

The space $\cN(V)$ is endowed with a partial order whereby $\norm\leq\norm'$ if, for all $v\in V$, $\norm[v]\leq\norm[v]'$. With respect to this partial order, $\cN(V)$ has the least upper bound property, which is left to the reader:
\begin{lemma}\label{lem:upperbound}
    Let $S\subset \cN(V)$ be a subset, and assume that there exists a finite-valued function $f:V\to\bbr$ such that, for all $\norm\in S$ and all $v\in V$, $\norm[v]\leq f(v)$. Then
    $$\norm[v]_S:=\sup_{\norm\in S} \norm[v]$$
    is a norm in $\cN(V)$.
\end{lemma}
The pointwise maximum of two norms is therefore a norm, which we denote by $\norm\vee\norm'$.

The space of norms $\cN(V)$ is endowed with a natural distance which we will call the \textit{Goldman--Iwahori distance} \cite{goldmaniwahori}, given by
$$d_\infty(\norm,\norm'):=\sup_{v\in V-\{0\}} \left|\log\frac{\norm[v]'}{\norm[v]}\right|,$$
which is finite as all norms on $V$ are equivalent. It has the following very useful characterisation: $d_\infty(\norm,\norm')$ is the smallest constant $C>0$ such that for all $v\in V$,
\begin{equation}\label{eq:dinfty}
    e^{-C}\norm[v]\leq\norm[v]'\leq e^C\norm[v].
\end{equation}
The Goldman--Iwahori space $(\cN(V),d_\infty)$ is Cauchy complete (\cite[Proposition 1.8]{boueri}). As a consequence of the previous lemma and \eqref{eq:dinfty}, we see that $d_\infty$-bounded subsets have a supremum:
\begin{lemma}
    Let $S\subset \cN(V)$ be a subset, and assume there exists $C>0$ such that, for all $\norm\in S$, and for a (hence any) $\norm'\in \cN(V)$,
    $$d_\infty(\norm,\norm')\leq C.$$
    Then the pointwise supremum of norms in $S$ is a norm in $\cN(V)$.
\end{lemma}

We say that a basis $(v_i)_i$ of $V$ is \textit{orthogonal} or \textit{diagonalising} for $\norm\in \cN(V)$ if, for all $v=\sum a_i v_i\in V$, we have
$$\norm[v]=\max_i |a_i|\norm[v_i].$$
We denote by $\cN^\diag(V)$ the set of norms on $V$ admitting an orthogonal basis. In general, the inclusion $\cN^\diag(V)\subseteq \cN(V)$ is strict, but we have the following results (\cite[Lemma 1.12, Proposition 1.14, Theorem 1.19]{boueri}):
\begin{theorem}
    $\cN^\diag(V)$ is $d_\infty$-dense in $\cN(V)$. If $\dim V\leq 1$, $\cN^\diag(V)=\cN(V)$. If $\dim V\geq 2$, the following are equivalent:
    \begin{enumerate}
        \item $\cN^\diag(V)=\cN(V)$;
        \item any two norms in $V$ admit a joint orthogonal basis;
        \item $(\field,|\cdot|)$ is \textit{spherically complete}.
    \end{enumerate}
\end{theorem}
Spherical completeness means that any nested sequence of $|\cdot|$-balls in $\field$ has nonempty intersection. Importantly, any field is spherically complete with respect to the trivial absolute value, as is easily checked, so that:
\begin{theorem}
    Over any non-Archimedean field, $\cN^0(V)=\cN_0^\diag(V)$.
\end{theorem}
Local fields are also spherically complete, meaning that diagonalisability issues do not arise in \cite{goldmaniwahori,book:gerardin}. Furthermore, any non-Archimedean field admits a spherical completion, either into a Hahn-type field (a formal series field over the residue field with the support of the series being well-ordered sets of $\bbr$) or into a $p$-adic Mal'cev--Neumann field (the $p$-adic-type construction starting from a Hahn-type field), see \cite{poonen}.  This fact is important in view of the following construction:
\begin{definition}
    Consider a valued field extension $(\mathfrak{l},|\cdot|_\mathfrak{l})$ of $(\field,|\cdot|)$ and $\norm\in \cN(V)$, we define the \textit{ground field extension} of $\norm$ to be, for $w\in V\otimes_\field \mathfrak{l}$,
    $$\norm[w]_\mathfrak{l}:=\inf_{w=\sum b_i v_i,\,b_i\in \mathfrak{l},\,v_i\in V}\max_i |b_i|_\mathfrak{l}\norm[v_i].$$
\end{definition}
In particular, the restriction of $\norm_\mathfrak{l}$ to $V$ coincides with the initial norm $\norm$, and it is the largest norm on $V\otimes_\field \mfl$ with this property (\cite[Proposition 1.25(i)]{boueri}). From the definition, we also see that
\begin{align}\label{eq:fieldineq}
    \norm\leq\norm'\Leftrightarrow \norm_\mfl\leq\norm'_\mfl.
\end{align}
We will often use this fact to extend non-diagonalisable norms to a spherically complete field, where they can be diagonalised, use results on diagonalisable norms on the larger field, and restrict them back to the base field. In particular, the following will be helpful:
\begin{proposition}[{\cite[Proposition 1.25(ii)]{boueri}}]\label{prop:groundfieldisometry}
    The embedding
    \begin{align*}
        (\cN(V),d_\infty)&\hookrightarrow (\cN(V\otimes_\field \mathfrak{l}),d_\infty),\\
        \norm&\mapsto \norm_\mfl
    \end{align*}
    is isometric.
\end{proposition}

We will occasionally use the following lemma:
\begin{lemma}\label{lem:diagenvelope}
    Let $S\subset \cN(V)$ be a bounded above subset of norms, such that there exists a basis $(v_i)_i$ of $V$ which is orthogonal for all norms in $S$. Then the pointwise supremum $\norm_S$ also admits $(v_i)_i$ as an orthogonal basis.
\end{lemma}
\begin{proof}
    Let $j$ be the index realising $|a_j|\norm[v_j]_S=\max_i |a_i|\norm[v_i]_S$. Then for all $\norm\in S$,
    $$|a_j|\norm[v_j]\leq \max_i |a_i|\norm[v_i]=\norm[v]\leq \norm[v]_S,$$
    thus
    $$\max_i |a_i|\norm[v_i]_S=|a_j|\norm[v_j]_S=\sup_{\norm\in S}|a_j|\norm[v_j]\leq \sup_{\norm\in S}\norm[v]=\norm[v]_S,$$
    while the reverse inequality is simply the ultrametric inequality.
\end{proof}
In particular, if a basis is jointly orthogonal for two norms, then it is also jointly orthogonal for their pointwise maximum.

We conclude those preliminaries with a helpful formula from Goldman--Iwahori \cite[Proposition 2.1]{goldmaniwahori}. We do not rewrite their proof, and in fact we will prove a more general version of this result in Lemma \ref{lem:gilambdaortho}.
\begin{lemma}\label{lem:giortho}
    Let $\norm,\norm'\in \cN(V)$. Let $(v_i)$ be an orthogonal basis for $\norm$, $(w_i)$ an orthogonal basis for $\norm'$. Then, 
    $$d_\infty(\norm,\norm')=\max_{i,j}\left(\log\frac{\norm[v_i]'}{\norm[v_i]},\log\frac{\norm[w_i]}{\norm[w_i]'}\right).$$
    In particular, if $(v_i)_i$ is jointly orthogonal for both norms, then
    $$d_\infty(\norm,\norm')=\max_{i}\left|\log\frac{\norm[v_i]'}{\norm[v_i]}\right|.$$
\end{lemma}

\subsection{Distances, volumes, determinants.}\label{sect:12}

In the Hermitian case, the eigenvalues of the transition matrix between two Hermitian norms on a complex vector space are characterised by the Courant--Fischer theorem. Analogously, we define:
\begin{definition}
    The \textit{successive minima} of two norms $\norm,\norm'\in \cN(V)$ are given by the decreasing sequence of real numbers
    $$\lambda_i(\norm,\norm')=\sup_{W\in Gr_\field(V,i)}\inf_{w\in W-\{0\}}\log\frac{\norm[w]'}{\norm[w]}.$$
\end{definition}
If $\norm,\norm'$ admit a joint diagonalising basis $(v_i)_i$, ordered so that the $\frac{\norm[v_i]'}{\norm[v_i]}$ are a decreasing sequence, then (\cite[Proposition 2.24]{boueri})
\begin{equation}
\lambda_i(\norm,\norm')=\log\frac{\norm[v_i]'}{\norm[v_i]}.
\end{equation}

We note the following, which we will use throughout the article.
\begin{proposition}\label{prop:succmin}
    The $\lambda_i$ are $1$-Lipschitz in each variable with respect to $d_\infty$; further,
    $$d_\infty(\norm,\norm')=\max_i |\lambda_i(\norm,\norm')|.$$
\end{proposition}

\begin{definition}
    We define the \textit{relative spectral measure} of $\norm,\norm'\in \cN(V)$ to be
    $$\sigma(\norm,\norm'):=n\mi \sum_{i=1}^n \delta_{\lambda_i(\norm,\norm')}.$$
\end{definition}

The following propositions are easily checked from Proposition \ref{prop:succmin}.

\begin{proposition}\label{prop:measure}
    Let $\norm,\norm'\in \cN(R)$, and $f:\bbr\to\bbr$.
    \begin{enumerate}
        \item if $f(\lambda)=-\lambda$, $f_*\sigma(\norm,\norm')=\sigma(\norm',\norm)$;
        \item if $f(\lambda)=\lambda+c$, then for any $a+b=c$, $f_*\sigma(\norm,\norm')=\sigma(e^{-b}\norm,e^a\norm')=\sigma(e^{-c}\norm,\norm')=\sigma(\norm,e^c\norm')$;
        \item if $f(\lambda)=\max(\lambda,0)$, then $f_*\sigma(\norm,\norm')=\sigma(\norm,\norm\vee\norm')$;
        \item if $f(\lambda)=\max(\lambda,c)$, then $f_*\sigma(\norm,\norm')=\sigma(\norm,(e^c \norm)\vee\norm').$
    \end{enumerate}
\end{proposition}

\begin{remark}
    In the diagonalisable case, we can pick an orthogonal basis $(v_i)_i$ and define a joint measure $\Sigma(\norm,\norm')=n\mi \sum_i \delta_{(-\log\norm[v_i],-\log\norm[v_i]')}$ on $\bbr^2$, following \cite[Section 3.1.3]{blxz}. Unlike in \cite{blxz} which treats the case of $\cN_0$, this joint measure depends on the choice of basis, but the pushforward by $f(x,y)=y-x$, which is exactly $\sigma(\norm,\norm')$, does not. This joint measure allows us to encode the relative measure of the geodesic joining both norms, as in the next section, and to notice that if $\sigma(e^c\norm,e^{c}\norm')=f_*\Sigma(\norm,\norm')=\sigma(\norm,\norm')$ using that $f(x,y)=y-x=y+c-x-c$.
\end{remark}

\begin{definition}
    The \textit{relative volume} of $\norm,\norm'\in \cN(V)$ is defined as
    $$\vol(\norm,\norm'):=\int_\bbr x\,d\sigma(\norm,\norm')(x).$$
    In the case where the norms are in $\cN_0(V)$ and $\norm_0'=\norm_\triv$, we usually write
    $$\vol(\norm_0):=\vol(\norm_0,\norm_\triv).$$
\end{definition}

The volume has an alternative characterisation using the determinant. Consider $\norm\in \cN(V)$, and the induced norm $\det \norm$ on $\det V$, defined as follows: let $\delta\in \det V$. We then set
$$\det\norm[\delta]:=\inf_{\delta=v_1\wedge\dots\wedge v_n}\prod_{i=1}^n \norm[v_i].$$
This norm coincides with the quotient norm of the symmetric power norm induced by $\norm$ on the adequate symmetric power on $V$.
In particular, if $(v_i)_i$ is orthogonal for $\norm$, then (\cite[Lemma 2.5]{boueri})
$$\det\norm[v_1\wedge\dots\wedge v_n]=\prod_{i=1}^n \norm[v_i].$$
The relative volume may then be expressed as follows:
\begin{theorem}[{\cite[Theorem 2.25]{boueri}}]\label{thm:voldet}
    Given $\norm,\norm'\in V$, we have that for all nonzero $\delta\in \det V$,
    $$\vol(\norm,\norm')=\log\frac{\det\norm[\delta]'}{\det\norm[\delta]},$$
    equivalently: the volume of two norms is the (unique) successive minimum of their determinants.
\end{theorem}

From this, we can deduce the \textit{cocycle formula} for the relative volume. The other properties below follow or are easily checked.
\begin{proposition}\label{prop:vol}
    Let $\norm,\norm',\norm''\in \cN(V)$, $\norm_0,\norm'_0\in \cN_0(V)$.
    \begin{enumerate}
        \item $\vol(\norm,\norm')=\vol(\norm,\norm'')+\vol(\norm'',\norm')$;
        \item $\vol(\norm,\norm')=-\vol(\norm',\norm)$;
        \item $\vol$ is $1$-Lipschitz with respect to $d_\infty$ in each variable;
        \item $\vol(e^c\norm,e^d\norm')=(d-c)+\vol(\norm,\norm');$
        \item $\vol$ is decreasing in the first variable and increasing in the second variable with respect to the partial order $\leq$;
        \item $\vol(t.\norm_0,s.\norm'_0)=s\cdot \vol(\norm_0,\norm'_0)+(s-t)\cdot \vol(\norm_0)=t\cdot \vol(\norm_0,\norm'_0)+(s-t)\vol(\norm'_0)$.
    \end{enumerate}
\end{proposition}

\begin{definition}
    We define the \textit{$d_p$ distance}, for $p\in [1,\infty)$, as
    $$d_p(\norm,\norm')^p:=\int_\bbr |x|^p\,d\sigma(\norm,\norm')(x)=n\mi\sum_i |\lambda_i(\norm,\norm')|^p.$$
    This is indeed a distance on $\cN(V)$ by \cite[Theorem 3.1]{boueri}. Note that we always have $d_1\leq d_p \leq d_\infty$.
\end{definition}

We have the following Darvas-type expression for $d_1$, reducing it to a volume expression:
\begin{proposition}\label{prop:d1vol}
    $$d_1(\norm,\norm')=\vol(\norm,\norm\vee\norm')+\vol(\norm',\norm\vee\norm').$$
    In particular, if $\norm\leq\norm'$, then
    $$d_1(\norm,\norm')=\vol(\norm,\norm').$$
\end{proposition}
\begin{proof}
    Assuming $\norm,\norm'\in \cN^\diag(V)$, we pick a jointly orthogonal basis $(v_i)_i$ for both norms. By Lemma \ref{lem:diagenvelope}, it is also orthogonal for $\norm\vee\norm'$, and further
    $\lambda_i(\norm,\norm\vee\norm')$ is $0$ iff $\norm[v_i]\geq \norm[v_i]'$. Summing everything together yields the equality. The general case follows by approximation using density of diagonalisable norms, since $d_1$ and $\vol$ are $d_\infty$-Lipschitz, and the pointwise maximum is easily seen to be $d_\infty$-continuous in each variable.
\end{proof}
In the case of trivially-valued norms, using our previous notation omitting $\norm_\triv$ and the cocycle formula, the expression above is equivalent to
\begin{equation}
    d_1(\norm_0,\norm'_0)=\vol(\norm_0)+\vol(\norm'_0)-2\vol(\norm_0\vee\norm'_0).
\end{equation}

Finally, we note a result that will be useful later on, which proof we recall for the convenience of the reader:
\begin{proposition}[{\cite[Lemma 2.6.3]{reb:2}}]\label{prop:fieldextdp}
    Let $(\mfl,|\cdot|_\mfl)$ be a valued field extension of $(\field,|\cdot|)$. Then, ground field extension preserves successive minima. In particular, for all $p$,   
    $$(\cN(V),d_p)\to (\cN(V\otimes_\field\mfl),d_p)$$ is an isometric embedding.
\end{proposition}
\begin{proof}
    It is clear that ground field extension of $\norm\in \cN^\diag(V)$ preserves successive minima. Now, successive minima are $d_\infty$-Lipschitz, and ground field extension is an isometry by Proposition \ref{prop:groundfieldisometry}. Thus, $d_\infty$-approximating $\norm,\norm'\in \cN(V)$ by $\norm_k,\norm'_k\in \cN^\diag(V)$ yields
    \begin{align*} 
        &|\lambda_i(\norm,\norm')-\lambda_i(\norm_\mfl,\norm'_\mfl)|\\
        &\leq |\lambda_i(\norm,\norm')-\lambda_i(\norm_k,\norm'_k)|+|\lambda_i(\norm_k,\norm'_k)-\lambda_i(\norm_\mfl,\norm'_\mfl)|\\
        &=|\lambda_i(\norm,\norm')-\lambda_i(\norm_k,\norm'_k)|+|\lambda_i(\norm_{k,\mfl},\norm'_{k,\mfl})-\lambda_i(\norm_\mfl,\norm'_\mfl)|\\
        &\leq d_\infty(\norm,\norm_k)+d_\infty(\norm',\norm_k')+d_\infty(\norm_\mfl,\norm_{k,\mfl})+d_\infty(\norm'_\mfl,\norm'_{k,\mfl})\\
        &= 2d_\infty(\norm,\norm_k)+2d_\infty(\norm',\norm'_k)\to_{k\to\infty}0.
    \end{align*}
\end{proof}

\subsection{The action of $\cN_0(V)$.}\label{sect:13}

In this section, we review a construction of Gérardin (that we extend from local fields to general non-Archimedean fields) which yields an action of $\cN_0(V)$ on $\cN(V)$, called \textit{geodesic action} in \cite{book:gerardin}.

\begin{definition}
    Let $\norm\in \cN(V)$, $\norm_0\in \cN_0(V)$. We define, for $v\in V$,
    $$\norm_0.\norm[v]:=\sup\{\norm[v]',\,\norm'\in \cN(V),\,\norm[v']'\leq \norm[v']\cdot \norm[v']_0\,\forall v'\in V\}.$$
\end{definition}
Namely, the product of a norm in $\cN(V)$ with a norm in $\cN_0(V)$ is no longer a norm (unless $\dim V\leq 1)$, but only a positive definite $|\cdot|$-homogeneous function. Seeing such a function as a gauge, we can construct the largest norm bounded above by this gauge. By Lemma \ref{lem:upperbound} this operation does define a norm provided the set we take the supremum over is nonempty. Further, it is easily seen that, given $c\in \bbr$, for all $a+b=c$, we have 
\begin{equation}\label{eq:sum}
    e^c(\norm_0.\norm)=(e^a\norm_0).(e^b\norm).
\end{equation}

\begin{remark}
    Picking $|\cdot|=|\cdot|_0$ we see that $\cN_0(V)$ acts on itself. While it is clearly commutative, a counterexample of Gérardin \cite[pp.149-150]{book:gerardin} shows that this operation is associative iff $\dim V\leq 1$.
\end{remark}

We first recall the following.
\begin{theorem}
    If $\norm\in \cN^\diag(V)$, then for all $\norm_0\in \cN_0(V)$ there exists a basis $(v_i)_i$ of $V$ which is jointly orthogonal for $\norm$ and $\norm_0$.
\end{theorem}
\begin{proof}
    The data of $\norm_0$ is equivalent to the data of a weighted flag of subspaces of $V$, so that it suffices to show one can pick an orthogonal basis for $\norm$ adapted to that flag, which follows from \cite[2.3.3, Proposition, p. 162]{book:gerardin} (this is essentially the joint diagonalisation argument in $\cN(V)$ due to Goldman--Iwahori \cite{goldmaniwahori}, see also \cite[Proposition 1.14]{boueri}).
\end{proof}

\begin{proposition}\label{prop:envdiag}
    Assume there exists a basis $(v_i)_i$ orthogonal for $\norm$ and $\norm_0$. Then for all $v=\sum a_i v_i$,
    $$\norm_0.\norm[v]=\max_i |a_i|\norm[v_i]\norm[v_i]_0.$$
\end{proposition}
\begin{proof}
    It is clear that the above expression is a candidate for the envelope defining $\norm_0.\norm$, as it is a norm, and
    $$\max_i |a_i|\norm[v_i]\norm[v_i]_0=\max_i |a_i||a_i|_0\norm[v_i]\norm[v_i]_0\leq (\max_i |a_i|\norm[v_i])(\max_i |a_i|_0\norm[a_i]_0)=\norm[v]\norm[v]_0.$$
    Now pick $\norm'\leq \norm\norm_0$ pointwise. It will suffice to show $\norm'\leq \max_i |a_i|\norm[v_i]\norm[v_i]_0$. But by the ultrametric inequality,
    $$\norm[v]'\leq \max_i |a_i|\norm[v_i]'\leq \max_i |a_i|\norm[v_i]\norm[v_i]_0$$
    and the result follows.
\end{proof}

We now prove an alternative expression for the $\cN_0$ action, which resembles the Archimedean construction in \cite[(2.21)]{finski:sectionrings}.

\begin{lemma}\label{lem:infenvelope}
    Let $\norm_0\in \cN_0(V)$, $\norm\in \cN(V)$. Then
    $$\norm_0.\norm[v]=\inf_{v=\sum_k v_k}\max_k \norm[v_k]_0\cdot\norm[v_k].$$
\end{lemma}
\begin{proof}
    Let $\norm'$ be a candidate for the envelope defining $\norm_0.\norm$. Pick a decomposition $v=\sum_k v_k$. Then
    $$\norm[v]'\leq \max_k \norm[v_k]'\leq \max_k \norm[v_k]_0\norm[v_k].$$
    Since $\norm':=\norm_0.\norm$ is itself a candidate, we may take the infimum over such decompositions, which proves the $\leq$ inequality.

    Conversely, it is clear that the right-hand side defines a norm in $\cN(V)$, for any two decomposition $v=\sum_k v_k$ and $w=\sum_\ell w_\ell$ yield a decomposition of $v+w$. Now, picking the decomposition $v=v$ shows the right-hand side to be a candidate for the envelope defining $\norm_0.\norm$, proving the $\geq$ direction.
\end{proof}

Finally, we see how the action behaves under field extension (and restriction). This result will not be used, but shows that many problems involving envelopes of non-diagonalisable norms cannot always be tackled by passing to the diagonalisable case \textit{via} spherically complete extensions.

\begin{proposition}\label{prop:envelopeground}
    Let $\norm\in \cN(V)$, $\norm_0\in \cN_0(V)$. Consider a field extension $(\mfl,|\cdot|_\mfl)$ of $(\field,|\cdot|)$. Then $$\norm_{0,\mfl}.\norm_{\mfl}\leq (\norm_0.\norm)_\mfl,$$ with equality if $\norm\in \cN^\diag(V)$.

    Conversely, given $\norm_0^\mfl$ and $\norm^\mfl$, then
    $$(\norm_0^\mfl.\norm^\mfl)|_\field\leq (\norm_0^\mfl)|_\field.\norm^\mfl|_\field.$$
\end{proposition}

\begin{proof}
    Given $\norm'\in \cN(V\otimes_\field \mfl)$ bounded above by $\norm_{0,\mfl}\norm_\mfl$ then in particular its restriction $\norm|_\field$ to $V$ is bounded above by $\norm_{0,\mfl}|_\field\norm_\mfl|_\field=\norm_0\norm$. Thus $(\norm_{0,\mfl}.\norm_\mfl)|_\field\leq \norm_0.\norm$, hence
    $$((\norm_{0,\mfl}.\norm_\mfl)|_\field)_\mfl\leq (\norm_0.\norm)_\mfl$$
    by \eqref{eq:fieldineq}. Furthermore $((\norm_{0,\mfl}.\norm_\mfl)|_\field)_\mfl$ is the largest norm restricting to $(\norm_{0,\mfl}.\norm_\mfl)|_\field$ on $\field$, and $(\norm_{0,\mfl}.\norm_\mfl)$ is such a norm, hence
    \begin{equation}
        \norm_{0,\mfl}.\norm_\mfl\leq (\norm_0.\norm)_\mfl.
    \end{equation}
    The reverse inequality is clear in the orthogonalisable case, as the base change of an orthogonal basis for $\norm$ remains orthogonal for $\norm_\mfl$ (\cite[Proposition 1.25(iv)]{boueri}).

    Now, for the second part of the proof, by the largest extension property we have $\norm_\mfl\leq ((\norm_\mfl)|_\field)_\mfl$ and likewise for $\norm_0$, so by the envelope definition and the first part of the proof,
    $$\norm_0^\mfl.\norm^\mfl\leq ((\norm_0^\mfl)|_\field)_\mfl.((\norm^\mfl)|_\field)_\mfl\leq (\norm_0^\mfl|_\field.\norm^\mfl|_\field)_\mfl.$$
    and restricting back to $\field$,
    $$(\norm_0^\mfl.\norm^\mfl)|_\field\leq \norm_0^\mfl|_\field.\norm^\mfl|_\field.$$
\end{proof}

\subsection{Continuity of the action in $d_\infty$.}\label{sect:14}

We are now interested in continuity properties of the envelope construction, beginning with the case $d_\infty$. A tough problem arises in considering approximation from the diagonalisable case. As we have seen above, the action of $\cN_0$ doesn't behave well with respect to ground field extension, so that we need to work with non-diagonalisable norms. We begin with a few technical results on almost-orthogonality of norms.

\begin{definition}
    Given $\lambda\in [0,1]$, we say that a basis $(v_i)_i$ of $V$ is \textit{$\lambda$-orthogonal} for some $\norm\in \cN(V)$ if for all $v=\sum a_i v_i\in V$,
    $$\lambda \max_i |a_i|\norm[v_i]\leq \norm[v].$$
    In particular, a $1$-orthogonal basis is orthogonal, and all bases are $0$-orthogonal for all norms.
\end{definition}

\begin{lemma}
    Let $\norm,\norm'\in \cN(V)$. Assume that $(v_i)_i$ is $\lambda$-orthogonal for $\norm'$, with $0<\lambda\leq 1$. Then for all $\norm\in \cN(V)$,
    $$\max_i \frac{\norm[v_i]}{\norm[v_i]'}\leq\sup_{v\in V-\{0\}}\frac{\norm[v]}{\norm[v]'}\leq \lambda\mi \max_i \frac{\norm[v_i]}{\norm[v_i]'}.$$
    In particular, if $\lambda=1$, equality holds.
\end{lemma}
\begin{proof}
    It is clear that 
	$$\max_i \frac{\norm[v_i]}{\norm[v_i]}\leq \sup_{v\in V-\{0\}}\frac{\norm[v]}{\norm[v]'}.$$
	Now, pick $v=\sum a_i v_i$. Then
	\begin{align*}
	\norm[v]\leq \max_i |a_i|\norm[v_i]&=\max_i |a_i|\frac{\norm[v_i]}{\norm[v_i]'}\norm[v_i]'\\
	&\leq \left(\max_i  \frac{\norm[v_i]}{\norm[v_i]'}\right)\cdot \max_i |a_i|\norm[v_i]'\\
	&\leq \left(\max_i  \frac{\norm[v_i]}{\norm[v_i]'}\right)\cdot \lambda\mi \norm[v]'.
	\end{align*}
    Dividing by $\norm[v]'$ and taking the supremum over $v$ proves the result.
\end{proof}

We now refine Lemma \ref{lem:giortho} for $\lambda$-orthogonal bases.
\begin{lemma}\label{lem:gilambdaortho}
    Let $\norm,\norm'\in \cN(V)$. Given $0<\lambda,\mu<1$, let $(v_i)_i$ be $\lambda$-orthogonal for $\norm'$ and $(w_i)_i$ be $\mu$-orthogonal for $\norm$. Then,
    $$\log\max_{i,j}\left(\frac{\norm[v_i]}{\norm[v_i]'},\frac{\norm[w_i]'}{\norm[w_i]}\right)\leq d_\infty(\norm,\norm')\leq \log\max_{i,j}\left(\lambda\mi\frac{\norm[v_i]}{\norm[v_i]'},\mu\mi\frac{\norm[w_i]'}{\norm[w_i]}\right).$$
\end{lemma}
\begin{proof}
    We have by the previous lemma that
    $$\max_i \frac{\norm[v_i]}{\norm[v_i]'}\leq\sup_{v\in V-\{0\}}\frac{\norm[v]}{\norm[v]'}\leq \lambda\mi \max_i \frac{\norm[v_i]}{\norm[v_i]'}$$
    and
    $$\max_i \frac{\norm[w_i]'}{\norm[w_i]}\leq\sup_{v\in V-\{0\}}\frac{\norm[v]'}{\norm[v]}\leq \mu\mi \max_i \frac{\norm[w_i]'}{\norm[w_i]},$$
    whence the result.
\end{proof}

\begin{lemma}\label{lem:lambdaorthoenvelope}
    Assume that $(v_i)_i$ is $\mu$-orthogonal for $\norm_0\in \cN_0(V)$ and $\lambda$-orthogonal for $\norm\in \cN(V)$, $0<\lambda,\mu\leq 1$. Then for all $v=\sum a_i v_i$,
    $$\lambda\mu\max_i |a_i|\norm[v_i]_0\norm[v_i]\leq \norm_0.\norm[v].$$
\end{lemma}
\begin{proof}
    Given $v=\sum a_i v_i\in V$,
    \begin{align*}
        \lambda\mu\max_i |a_i|\norm[v_i]_0\norm[v_i]&=\lambda\mu\max_i |a_i|\norm[v_i]|a_i|_0\norm[v_i]_0\\
        &\leq \lambda\max_i(|a_i|\norm[v_i])\mu(\max_i |a_i|_0\norm[v_i]_0)\\
        &\leq\norm[v]\cdot \norm[v]_0.
    \end{align*}
    Since $\lambda\mu>0$, the expression $\lambda\mu\max_i |a_i|\norm[v_i]_0\norm[v_i]$ is indeed a norm, thus by the above inequality it is a candidate for the envelope defining $\norm_0.\norm$, hence it has to be smaller than $\norm_0.\norm$, concluding the proof.
\end{proof}

The following result, useful on its own, also proves when combined with density of diagonalisable norms that any norm admits a $\lambda$-orthogonal basis for all $0\leq \lambda<1$.

\begin{lemma}\label{lem:orthodinfty}
    Let $(v_i)_i$ be an orthogonal basis for $\norm\in \cN(V)$, and let $\norm'\in \cN(V)$. Then $(v_i)_i$ is $e^{-2d_\infty(\norm,\norm')}$-orthogonal for $\norm'$.
\end{lemma}
\begin{proof}
    Write $C=d_\infty(\norm,\norm')$ so that, for all $v\in V$,
    $$e^{-C}\norm[v]'\leq \norm[v]\leq e^{C}\norm[v].$$
    In particular, for all $i$, $e^{-C}\norm[v_i]'\leq \norm[v_i]$
    so that
    $$e^{-C}\max_i |a_i|\norm[v_i]'\leq \max_i |a_i|\norm[v_i]=\norm[v]\leq e^{C}\norm[v]',$$
    proving the result.
\end{proof}

It also has as consequence the following:
\begin{corollary}\label{coro:jointlambdadiag}
    Given $\norm_0\in \cN_0(V)$ and $\norm\in \cN(V)$, for all $\lambda<1$, there exists a basis $(v_i)_i$ that is jointly orthogonal for $\norm_0$ and $\lambda$-orthogonal for $\norm$.
\end{corollary}
\begin{proof}
    Pick $\varepsilon>0$ and $\norm'\in \cN^\diag(V)$ with $d_\infty(\norm,\norm')<\varepsilon$. Pick a basis $(v_i)_i$ jointly orthogonal for $\norm_0$ and $\norm'$. Then, by Lemma \ref{lem:orthodinfty} this basis is $e^{-2\varepsilon}$-orthogonal for $\norm$.
\end{proof}

We are then able to prove the main technical result of this section.
\begin{proposition}\label{prop:appxenvelope}
    Let $(\norm_k)_k$ be a sequence in $\cN^\diag(V)$ with $d_\infty(\norm_k,\norm)\to 0$. Then $d_\infty(\norm_0.\norm_k,\norm_0.\norm)\to 0$.
\end{proposition}
\begin{proof}
    For each $k$, pick a joint orthogonal basis $(v_{k,i})_i$ of $V$ for $\norm_0$ and $\norm_k$, which is then also orthogonal for $\norm_0.\norm_k$, and $\norm_0.\norm[v_{k,i}]_k=\norm[v_{k,i}]_0\norm[v_{k,i}]_k$. This basis is $\lambda_k:=e^{-2d_\infty(\norm,\norm_k)}$-orthogonal for $\norm$ by Lemma \ref{lem:orthodinfty} hence also for $\norm_0.\norm$ by Lemma \ref{lem:lambdaorthoenvelope}. Further, by Lemma \ref{lem:gilambdaortho} we have that
    $$d_\infty(\norm_0.\norm_k,\norm_0.\norm)\leq \log\lambda_k\mi\max_i\left|\frac{\norm[v_{k,i}]_0\norm[v_{k,i}]_k}{\norm_0.\norm[v_{k,i}]}\right|.$$
    Now on the one hand, by Lemma \ref{lem:orthodinfty} and Lemma \ref{lem:lambdaorthoenvelope} again we have that $\norm_0.\norm[v_{k,i}]\geq \lambda_k\norm[v_{k,i}]_0\norm[v_{k,i}]$ so that
    $$\log\frac{\norm[v_{k,i}]_0\norm[v_{k,i}]_k}{\norm_0.\norm[v_{k,i}]}\leq \log\lambda_k\mi\frac{\norm[v_{k,i}]_k}{\norm[v_{k,i}]},$$
    while by the definition of the envelope, $\norm_0.\norm[v_{k,i}]\leq \norm[v_{k,i}]_0\cdot\norm[v_{k,i}]$ and thus
    $$\log\frac{\norm_0.\norm[v_{k,i}]}{\norm[v_{k,i}]_0\norm[v_{k,i}]_k}\leq \log \frac{\norm[v_{k,i}]}{\norm[v_{k,i}]_k}.$$
    Putting everything together, we then have
    $$d_\infty(\norm_0.\norm_k,\norm_0.\norm)\leq\log\lambda_k^{-2}\max_i\left|\frac{\norm[v_{k,i}]_k}{\norm[v_{k,i}]}\right|\leq \log\lambda_k^{-2}+d_\infty(\norm,\norm_k),$$
    which is equal to $3d_\infty(\norm,\norm_k)$.
\end{proof}

\begin{theorem}\label{thm:dinftyenvelope}
    For all $\norm_0,\norm'_0\in \cN_0(V)$, $\norm,\norm'\in \cN(V)$, we have
    $$d_\infty(\norm_0.\norm,\norm'_0.\norm')\leq d_\infty(\norm_0,\norm'_0)+d_\infty(\norm,\norm').$$
\end{theorem}
\begin{proof}
    We begin with the diagonalisable case. Pick a jointly diagonalising basis $(v_i)_i$ for $\norm_0$ and $\norm$, and a jointly diagonalising basis $(w_i)_i$ for $\norm'_0$ and $\norm'$. Then using the envelope characterisation (for the $\norm'$s) and Proposition \ref{prop:envdiag} (for the $\norm$s) we have that
    \begin{align*}
        \log \frac{\norm'_0.\norm[v_i]'}{\norm_0.\norm[v_i]}&\leq \log \frac{\norm[v_i]'_0\cdot\norm[v_i]'}{\norm[v_i]_0\cdot\norm[v_i]}\\
        &=\log\frac{\norm[v_i]'_0}{\norm[v_i]}+\log\frac{\norm[v_i]'}{\norm[v_i]}.
    \end{align*}
    Likewise,
    \begin{align*}
        \log\frac{\norm_0.\norm[w_i]}{\norm'_0.\norm[w_i]'}\leq \log\frac{\norm[w_i]_0}{\norm[w_i]'_0}+\log\frac{\norm[w_i]}{\norm[w_i]'}.
    \end{align*}
    The result then follows from Lemma \ref{lem:giortho} and the triangle inequality. In the general case, we approximate $\norm$ and $\norm'$ by sequences $(\norm_k)_k,(\norm'_k)_k$ in $\cN^\diag(V)$. Then
    \begin{align*}
        &d_\infty(\norm_0.\norm,\norm'_0.\norm')\\
        &\leq d_\infty(\norm_0.\norm,\norm_0.\norm_k)+d_\infty(\norm_0.\norm_k,\norm'_0.\norm_k)+d_\infty(\norm'_0.\norm_k,\norm'_0.\norm')\\
        &\leq d_\infty(\norm_0.\norm,\norm_0.\norm_k)+d_\infty(\norm_0,\norm'_0)+d_\infty(\norm_k,\norm'_k)+d_\infty(\norm'_0.\norm_k,\norm'_0.\norm').
    \end{align*}
    In the last inequality, we have used the first part of the proof, given that $\norm_k$ and $\norm'_k$ are in $\cN^\diag(V)$. By Proposition \ref{prop:appxenvelope}, the first and last term vanish as $k\to\infty$, while the middle terms converge to the desired expression. This concludes the proof.
\end{proof}

\subsection{Continuity of the action in $d_1$.}\label{sect:15}

For $d_1$, we first need an auxiliary result on the volume. Note that, in dimension $1$, the product of a norm in $\cN(V)$ and $\cN_0(V)$ is well-defined, thus $\det\norm_0.\det\norm=\det\norm_0\cdot\det\norm$. In \cite[p.167]{book:gerardin} the following was observed over a local field, and is now easily extended to the non-diagonalisable case:
\begin{lemma}
    As norms in $\cN(\det V)$, we have that
    $$\det(\norm_0.\norm)=\det\norm_0.\det\norm=\det\norm_0\cdot\det\norm.$$
\end{lemma}
\begin{proof}
    If $\norm\in \cN^\diag(V)$, we pick a jointly orthogonalising basis $(v_i)_i$ for $\norm_0$ and $\norm$, which is then also orthogonal for $\norm_0.\norm$, so that
    $$\det\norm[v_1\wedge\dots\wedge v_n]=\prod_i \norm[v_i],\quad\det\norm[v_1\wedge\dots\wedge v_n]_0=\prod_i \norm[v_i]_0,$$
    $$\det\norm_0.\norm[v_1\wedge\dots\wedge v_n]=\prod_i \norm[v_i]_0\norm[v_i],$$
    and the result follows. In general, we $d_\infty$-approximate $\norm$ by a sequence $(\norm_k)_k\in \cN^\diag(V)$. Then
    \begin{align*}
        &d_\infty(\det(\norm_0.\norm),\det\norm_0\cdot\det\norm)\\
        &\leq d_\infty(\det(\norm_0.\norm),\det(\norm_0.\norm_k))+d_\infty(\det(\norm_0.\norm_k),\det\norm_0\cdot\det\norm_k)\\
        &+d_\infty(\det\norm_0.\det\norm_k,\det\norm_0.\det\norm)\\
        &\leq nd_\infty(\norm_0.\norm,\norm_0.\norm_k)+d_\infty(\det(\norm_0.\norm_k),\det\norm_0\cdot\det\norm_k)\\
        &+d_\infty(\det\norm_k,\det\norm)\\
        &\leq nd_\infty(\norm,\norm_k)+d_\infty(\det(\norm_0.\norm_k),\det\norm_0\cdot\det\norm_k)+nd_\infty(\norm_k,\norm).
    \end{align*}
    Here, we have used that $\det$ is $n$-Lipschitz as is seen from the definition, as well as Theorem \ref{thm:dinftyenvelope}. The middle term vanishes by the first part of the proof, proving the result.
\end{proof}

From Theorem \ref{thm:voldet} we then have that:
\begin{corollary}\label{coro:volenvelope}
    Given $\norm,\norm'\in \cN(V)$, $\norm_0,\norm'_0\in \cN_0(V)$, we have that
    $$\vol(\norm_0.\norm,\norm'_0.\norm')=\vol(\norm_0,\norm'_0)+\vol(\norm,\norm').$$
\end{corollary}

This now allows us to prove our desired estimate.
\begin{theorem}\label{thm:d1envelope}
    Given $\norm,\norm'\in \cN(V)$, $\norm_0,\norm'_0\in \cN_0(V)$, we have that
    $$d_1(\norm_0.\norm,\norm'_0.\norm')\leq d_1(\norm_0,\norm'_0)+d_1(\norm,\norm').$$
\end{theorem}
\begin{proof}
    Assume first that $\norm'\geq\norm$ and $\norm'_0\geq\norm_0$. Then $\norm'_0.\norm'\geq \norm_0.\norm$ by the envelope definition, so all $d_1$'s in the statement of the proposition are computed by $\vol$ by Proposition \ref{prop:d1vol}, hence the result holds (with equality) by Corollary \ref{coro:volenvelope}.

    We now observe that $$(\norm_0.\norm)\vee(\norm_0'.\norm')\leq (\norm_0\vee\norm'_0).(\norm\vee\norm').$$
    Indeed, let $\norm_*\leq \norm_0\norm$ and $\norm'_*\leq \norm_0'\norm'$. Then, 
    $$\norm_*\leq (\norm_0\vee\norm'_0)\cdot(\norm\vee\norm'),\quad \norm'_*\leq (\norm_0\vee\norm'_0)\cdot (\norm\vee\norm')$$
    hence
    $$\norm_*\vee\norm'_*\leq (\norm_0\vee\norm'_0)\cdot (\norm\vee\norm').$$
    In particular, we can pick $\norm_*:=\norm_0.\norm$ and $\norm'_*:=\norm'_0.\norm'$. Thus their pointwise maximum is a candidate for the envelope defining $(\norm_0\vee\norm'_0).(\norm\vee\norm')$, hence it is smaller. Using this, the fact that $\vol$ is increasing in the second variable, and repeatedly Proposition \ref{prop:d1vol}, we find
    \begin{align*}
        &d_1(\norm_0.\norm,\norm'_0.\norm')\\
        &=\vol(\norm_0.\norm, (\norm_0.\norm)\vee(\norm_0'.\norm'))+\vol(\norm_0'.\norm',(\norm_0.\norm)\vee(\norm_0'.\norm'))\\
        &\leq \vol(\norm_0.\norm, (\norm_0\vee\norm'_0).(\norm\vee\norm'))+\vol(\norm_0'.\norm',(\norm_0\vee\norm'_0).(\norm\vee\norm'))\\
        &=\vol(\norm_0,\norm_0\vee \norm'_0)+\vol(\norm,\norm\vee\norm')+\vol(\norm'_0,\norm_0\vee\norm_0')+\vol(\norm',\norm\vee\norm')\\
        &=d_1(\norm_0,\norm_0')+d_1(\norm,\norm').
    \end{align*}
\end{proof}

\begin{remark}\label{rem:diag}
    If $\norm_0,\norm'_0$ and $\norm$ admit a joint diagonalising basis $(v_i)_i$, then
    $$\lambda_i(\norm_0.\norm,\norm'_0.\norm)=\log\frac{\norm[v_i]'_0\cdot\norm[v_i]}{\norm[v_i]_0\cdot\norm[v_i]}=\lambda_i(\norm_0,\norm'_0).$$
    That is, equality holds not only at the level of volumes, but also at the level of each successive minima. In particular, for all $p$, we have
    $$d_p(\norm_0.\norm,\norm'_0.\norm)=d_p(\norm_0,\norm'_0)$$
    in this case. Of course, it is not the case in general (unless $\dim V\leq 1$) that such a basis exists, but it will turn out to be useful in constructions where we pick a joint diagonalising basis for two norms and are allowed to modify freely one of the two norms along this basis, as is the case for Theorem C.
\end{remark}

\subsection{Geodesic segments.}\label{sect:16}

The scaling action of $\bbr_{>0}$ on $\cN_0(V)$ combined with the Gérardin action of $\cN_0(V)$ turns out to give geodesic rays in $\cN(V)$. Before studying those, we first review and improve a few results from \cite{reb:2}, \cite{goldmaniwahori}.

\begin{definition}\label{def:geod}
    Given two norms $\norm,\norm'\in \cN(V)$, we define the \textit{geodesic} joining them to be given at time $t\in[0,1]$ by the envelope
    $$\norm[v]_t:=\sup\{\norm[v]_*,\norm_*\in\cN(V),\,\norm[v']_*\leq\norm[v']^{1-t}\norm[v']'{}^t\,\,\forall v'\in V\}.$$
    In the case where $\norm,\norm'\in \cN^\diag(V)$, and given a jointly orthogonal basis $(v_i)_i$ for them, it is further given by 
    $$\norm[\sum a_i v_i]_t:=\max_i |a_i|\norm[v_i]^{1-t}\norm[v_i]'{}^t.$$
\end{definition}

This envelope definition mirrors the one in \cite{goldmaniwahori} (where they were not called geodesics, but were used in the proof of contractibility and arcwise-connectedness of $\cN(V)$ over $\bbq_p$). In \cite{reb:2}, we defined geodesics in the diagonalisable case as above, and general geodesics to be abstract $d_\infty$-limits of diagonalisable geodesic segments, which converge by completeness. The Goldman--Iwahori definition is more practical for our purposes, and turns out to coincide with the one in \cite{reb:2}, as we now show in the geodesic analogue of Proposition \ref{prop:appxenvelope}.

\begin{proposition}\label{prop:geodappx}
    Given a sequence of geodesics $(t\mapsto \norm_{k,t})_k$ whose endpoints are in $\cN^\diag(V)$ and $d_p$-approximate $\norm_0$ and $\norm_1$, then for each $t\in [0,1]$, 
    \begin{equation}\label{eq:show}
        d_p(\norm_{k,t},\norm_t)\to 0.
    \end{equation}
\end{proposition}
\begin{proof}
    For each $k$, pick a joint orthogonal basis $(v_{k,i})$ for $\norm_{k,0}$ and $\norm_{k,1}$ which then also orthogonalises $\norm_{k,t}$ with $\norm[v_{k,i}]_{k,t}=\norm[v_{k,i}]_{k,0}^{1-t}\norm[v_{k,i}]_{k,1}^t$. Then, by Lemma \ref{lem:orthodinfty} this basis is $C_{k,0}:=e^{-2d_\infty(\norm_0,\norm_{k,0})}$-orthogonal for $\norm_0$ and $C_{k,1}:=e^{-2d_\infty(\norm_1,\norm_{k,1})}$-orthogonal for $\norm_1$. Analogously to Lemma \ref{lem:lambdaorthoenvelope} we see that, for $v=\sum a_i v_{k,i}$,
    \begin{align*} 
        C_{k,0}^{1-t}C_{k,1}^t\max_i |a_i|\norm[v_{k,i}]_0^{1-t}\norm[v_{k,i}]_1^t&\leq (C_{k,0}\max_i |a_i|\norm[v_{k,i}]_0)^{1-t} (C_{k,1}\max_i |a_i|\norm[v_{k,i}]_1)^t\\
        &\leq \norm[v]_0^{1-t}\norm[v]_1^t
    \end{align*}
    hence
    $$C_{k,0}^{1-t}C_{k,1}^t\max_i |a_i|\norm[v_{k,i}]_0^{1-t}\norm[v_{k,i}]_1^t\leq \norm_t,$$
    hence this basis is $C_{k,t}:=C_{k,0}^{1-t}C_{k,1}^t$-orthogonal for $\norm_t$.
    By Lemma \ref{lem:gilambdaortho} we then have
    $$d_\infty(\norm_t,\norm_{k,t})\leq \log C_{k,t}\mi\max_i \left|\frac{\norm[v_{k,i}]_{k,0}^{1-t}\norm[v_{k,i}]_{k,1}^t}{\norm[v_{k,i}]_t}\right|.$$
    Using that
    $$\log \frac{\norm[v_{k,i}]_{k,0}^{1-t}\norm[v_{k,i}]_{k,1}^t}{\norm[v_{k,i}]_t}\leq \log C_{k,t}\mi \frac{\norm[v_{k,i}]_{k,0}^{1-t}\norm[v_{k,i}]_{k,1}^t}{\norm[v_{k,i}]_{0}^{1-t}\norm[v_{k,i}]_{1}^t},$$
    $$\log\frac{\norm[v_{k,i}]_t}{\norm[v_{k,i}]_{k,0}^{1-t}\norm[v_{k,i}]_{k,1}^t}\leq \log \frac{\norm[v_{k,i}]_{0}^{1-t}\norm[v_{i}]_{k,1}^t}{\norm[v_{k,i}]_{k,0}^{1-t}\norm[v_{k,i}]_{k,1}^t}$$
    we then find that
    $$d_\infty(\norm_t,\norm_{k,t})\leq 3((1-t)d_\infty(\norm_0,\norm_{k,0})+td_\infty(\norm_1,\norm_{k,1}))\to_{k\to\infty} 0.$$

    Finally, for general $p$, this follows from the fact that all $d_p$ norms are equivalent, see below \cite[Theorem 3.1]{boueri}.
\end{proof}

\begin{theorem}\label{thm:busemann}
    Let $t\mapsto \norm_t,\norm'_t$ be two geodesics in $\cN(V)$. Then,
    $$d_\infty(\norm_t,\norm'_t)\leq (1-t)d_\infty(\norm_0,\norm'_0)+td_\infty(\norm_1,\norm'_1),$$
    $$d_1(\norm_t,\norm'_t)\leq (1-t)d_1(\norm_0,\norm'_0)+td_1(\norm_1,\norm'_1).$$
\end{theorem}
\begin{proof}
    Assuming all the endpoints belong to $\cN^\diag(V)$, then the $d_\infty$-version of the inequality is proven along the lines of Theorem \ref{thm:dinftyenvelope} using Lemma \ref{lem:giortho}, see also the proof of \cite[Proposition 2.6]{goldmaniwahori}. Much as the proof of Theorem \ref{thm:dinftyenvelope} in the non-diagonalisable case followed from Proposition \ref{prop:appxenvelope}, the result in the non-diagonalisable case follows from Proposition \ref{prop:geodappx}.

    The $d_1$ inequality was proven in the diagonalisable case in \cite[Corollary 2.4.8]{reb:2}. The general case follows from Proposition \ref{prop:geodappx} above.
\end{proof}

As is checked by computing successive minima, $\vol$ is affine along geodesics, and geodesic segments are indeed $d_1$ geodesics, in the diagonalisable case. They are not \textit{a priori} unique among $d_1$ geodesics, but they are uniquely geodesic for the $d_2$ distance (and in fact for $d_p$ distances, $p>1$) -- hence are \textit{distinguished} among $d_1$ geodesics. Thus, although we focus on $d_1$ in this article, this characterisation will turn out to be useful.
\begin{lemma}\label{lem:lambdageod}
    Given $\norm,\norm'\in \cN(V)$ and the geodesic segment $t\mapsto \norm_t$ from Definition \ref{def:geod} joining them, then for all $0\leq t,s\leq 1$,
    $$\vol(\norm_t,\norm_s)=(s-t)\vol(\norm,\norm')$$
    and for all $p\geq 1$,
    $$d_p(\norm_t,\norm_s)=|s-t|d_p(\norm,\norm').$$
\end{lemma}
\begin{proof}
    In the diagonalisable case, this is an immediate computation, and the result follows in the non-diagonalisable case from $d_\infty$-approximation using Proposition \ref{prop:geodappx} above.
\end{proof}

\begin{proposition}\label{prop:cat0}
    $(\cN(V),d_2)$ is a CAT(0) metric space. In particular, it is uniquely geodesic, with geodesic given by Definition \ref{def:geod}.
\end{proposition}
\begin{proof}
     By \cite[Example 3.2]{boueri}, $(\cN^\diag(V),d_2)$ is CAT(0), and it was seen in Lemma \ref{lem:lambdageod} that geodesics in Definition \ref{def:geod} are $d_2$-geodesics. The general result follows from approximating the CAT(0) inequality using density of diagonalisable norms and Proposition \ref{prop:geodappx}.
\end{proof}

\begin{corollary}\label{coro:geodesicbasechange}
    If $t\mapsto \norm_t$ is a geodesic, and $(\mfl,|\cdot|_\mfl)$ is a valued field extension of $(\field,|\cdot|)$, then $t\mapsto \norm_{t,\mfl}$ is a geodesic.
\end{corollary}
\begin{proof}
     This follows from Lemma \ref{lem:lambdageod}, Proposition \ref{prop:fieldextdp}, and the uniqueness part of Proposition \ref{prop:cat0}.
\end{proof}

Finally, the following result, akin to Berndtsson's maximum principle in the Hermitian case \cite{berndt:prob}, will be used throughout the entire article. With the envelope definition of geodesics, its proof is immediate.

\begin{lemma}[\textit{Maximum principle}, {\cite[Proposition 2.4.4]{reb:2}}]\label{lem:maxprinciple}
    Let $t\mapsto \norm_t,\norm'_t$ be two geodesics in $\cN(V)$. If $\norm'_0\leq \norm_0$ and $\norm'_1\leq\norm_1$, then for all $t\in [0,1]$,
    $$\norm'_t\leq\norm_t.$$
\end{lemma}

\subsection{Geodesic rays.}\label{sect:17}

\begin{definition}
    A \textit{geodesic ray} in $\cN(V)$ is a map $[0,\infty)\ni t\mapsto \norm_t\in \cN(V)$ whose restriction to any segment $[a,b]\subset [0,\infty)$ is a geodesic in the sense of Definition \ref{def:geod}.
\end{definition}

Now, as mentioned at the beginning of section \ref{sect:16}, we notice that given $\norm_0\in \cN_0(V)$, and $\norm\in \cN(V)$, then the ray
$$[0,\infty)\ni t\mapsto (t.\norm_0).\norm$$
restricts to a geodesic on each segment, hence is a geodesic ray. Our main result in this section is that those are essentially the only possible geodesic rays in $\cN(V)$. We begin with showing that the pointwise rescaled limit of a geodesic ray yields a norm in $\cN(V)$.

\begin{lemma}
    Let $\{\norm_t\}_t$ be a nonconstant geodesic ray in $\cN(V)$. Then $\lim_{t\to\infty}\norm_t^{\frac1t}\in \cN_0(V)$.
\end{lemma}
\begin{proof}
    By log-convexity of geodesics, this limit exists. It is then clear that it is nonnegative, and if it is finite, then it is also $|\cdot|_0$-homogeneous and satisfies the ultrametric inequality. We need only show it is finite and positive away from $0\in V$. Now, since $\{\norm_t\}_t$ is a geodesic ray, it is $d_\infty$-geodesic, hence $t\mapsto d_\infty(\norm_t,\norm_0)$ is linear, and $\lim_{t\to\infty} t\mi d_\infty(\norm_t,\norm_0)=s\mi d_\infty(\norm_s,\norm_0)=:C_\infty$ for any $s>0$. Since the ray is nonconstant, $C_\infty>0$. In particular, for all $v\in V-\{0\}$,
    $$0<e^{-C_\infty}\norm[v]_s^\frac1s\leq \lim_{t\to\infty}\norm[v]_t^\frac1t\leq e^{C_\infty}\norm[v]_s^\frac1s<\infty,$$
    proving the result.
\end{proof}

\begin{definition}
    Given a geodesic ray $\{\norm_t\}_t$ in $\cN(V)$, we define $\ell(\{\norm_t\}_t)\in\cN_0(V)$ to be the limit in the above lemma.
\end{definition}

Note in particular that
\begin{equation}\label{eq:uniq} 
\ell(\{(t.\norm_0).\norm\}_t)=\norm_0.
\end{equation}

We note the following for later use. It is clear from the definition as ground field extension preserves the values of norms on $V$.
\begin{lemma}
    Let $\{\norm_t\}_t$ be a geodesic ray in $\cN(V)$. Let $(\mfl,|\cdot|_\mfl)$ be a valued field extension, and consider the ground field extension $\{\norm_{t,\mfl}\}_t$. Then on $V$, we have
    $$\ell(\{\norm_t\}_t)=\ell(\{\norm_{t,\mfl}\}_t)$$
\end{lemma}

We now use a few preliminary definitions and results.
\begin{definition}
    Let $\underline v:=(v_i)_i$ be a basis of $v$. We define the \textit{apartment} $\bba_{\underline v}$ in $\cN^\diag(V)$ to be the set of all norms admitting $\underline v$ as an orthogonal basis.
\end{definition}

There is a map $\iota_{\underline v}:\bbr^n\hookrightarrow \bba_{\underline v}$ for all $\underline v$, given by sending $\alpha\in \bbr^n$ to
$$\iota_{\underline v}(\alpha)\left(\sum a_i v_i\right):=\max_i |a_i| e^{-\alpha_i}.$$
It is clear that this map is an injective isometry with respect to the usual $L^p$ norm on $\bbr^n$ and $d_p$ on $\cN(V)$. One can also change its basepoint (corresponding to the affine version of the above map), having fixed $\norm$ diagonalised in the basis $\underline v$, so that
$$\iota_{\underline v}(\alpha,\norm)\left(\sum a_i v_i\right):=\max_i |a_i|\norm[v_i]e^{-\alpha_i}.$$
In particular, apartments are characterised as being the image of a map of the form $\iota_{\underline v}$, and it is easily checked that geodesic rays starting at $\norm$ are the images of half-lines under the map $\iota_{\underline v}(\cdot,\norm)$. The following is a standard result in the theory of Bruhat--Tits buildings (\cite{kleinerleeb,parreau:immeubles}).

\begin{theorem}\label{thm:raystructure}
    If $\norm_t$ is a nonconstant geodesic ray of norms in $\cN^\diag(V)$, then it is of the form $(t.\ell(\{\norm_t\}_t)).\norm$.
\end{theorem}
\begin{proof}
    This follows as in \cite[Proposition 2.18]{parreau:immeubles} which shows that $t\mapsto \norm_t$ must lie in a same apartment for all $t$, hence is of the form $\iota_{\underline v}(t\alpha,\norm)$ for some basis $\underline v$. Picking $\norm_0\in \cN_0(V)$ defined as
    $$\norm[\sum a_i v_i]_0:=\max_i |a_i|_0 e^{-\alpha_i}$$
    then shows that $\norm_t=(t.\norm_0).\norm$, while it is clear that $\ell(\{\norm_t\}_t)=\norm_0$.
\end{proof}

We conclude with the following estimate. It is a finite-dimensional, and much simpler version, of the main result of our article.

\begin{theorem}\label{thm:convergencemeasurefindim}
    Let $t\mapsto \norm_t,\norm'_t$ be geodesic rays in $\cN(V)$. Then, for all $p\in [1,\infty]$,
    $$\lim_{t\to\infty}\frac{d_p(\norm_t,\norm'_t)}{t}=d_p(\ell(\{\norm_t\}_t),\ell(\{\norm_t\}_t)).$$
\end{theorem}
\begin{proof}
    Assume first that the considered rays lie in $\cN^\diag(V)$. By Theorem \ref{thm:raystructure} they are then of the form $(t.\ell(\{\norm_t\}_t)).\norm_0$, $(t.\ell(\{\norm'_t\}_t).\norm'_0)$. In this case, it is clear that the limit is independent of the starting point, by Theorem \ref{thm:dinftyenvelope}. We jointly diagonalise $\ell(\{\norm_t\}_t),\ell(\{\norm_t\}_t)$ and pick a norm $\norm\in \cN(V)$ also diagonalised in this basis. Without loss of generality, we may therefore assume that $\norm_t=(t.\ell(\{\norm_t\}_t)).\norm$ and that $\norm'_t=(t.\ell(\{\norm_t\}_t)).\norm$. The result then follows from Remark \ref{rem:diag}.

    In general, we pick a spherically complete field extension $(\mfl,|\cdot|_\mfl)$. The ground field extension of the given geodesic rays then lies in $\cN^\diag(V\otimes_\field \mfl)$, in which case the first part of the proof applies. Since ground field extension is an isometry by Proposition \ref{prop:fieldextdp}, the result follows.
\end{proof}

\begin{remark}
    The above very simple proof no longer works in the graded case considered in the next section: in general, it is not possible to construct a \textit{submultiplicative} norm which is diagonalised in a given sequence of bases, unless (e.g.) the algebra is toric.
\end{remark}

\section{Asymptotic geometry of the space of graded norms.}

Throughout, we let $X$ be a compact (connected) projective variety over $\field$, and we fix $L$ an ample line bundle on $X$. We write $R:=R(X,L)=\bigoplus_m H^0(X,mL)$, and for brevity $R_m=H^0(X,mL)$.

\subsection{First definitions.}\label{sect:21}

We define a \textit{graded norm} on $R$ to be the data $\norm_\bullet$, for each $m$, of a norm $\norm_m\in \cN(R_m)$ such that, for all $s_m\in R_m,s_sn\in R_n$, 
$$\norm[s_m\cdot s_n]_{m+n}\leq \norm[s_m]_m\norm[s_n]_n.$$
We say it is \textit{bounded} if there exists $C>0$ and a graded norm $\norm_\bullet'$ with
$$d_\infty(\norm_m,\norm'_m)\leq Cm$$
for all $m$. We denote by $\cN(R)$ the set of bounded graded norms on $R$. We likewise define the trivially-valued version of this space, denoted by $\cN_0(R)$. It contains the trivial graded norm $\norm_{\triv,\bullet}$ equal to $1$ on each $R_m-\{0\}$.

The space $\cN(R)$ inherits the pointwise partial order $\leq$ of norms, and the least upper bound property (hence the pointwise maximum of norms is well-defined), as we leave it to the reader to check. There is again a translating action of $\bbr$, $e^{\bullet c}\norm_\bullet$, given in degree $m$ by $e^{mc}\norm_m$. Finally, $\cN_0(R)$ also inherits the scaling action of $\bbr_{>0}$ given in degree $m$ by $(\norm_{0,m})^t$.


We define $\cN^\diag(R)$ to be the set of bounded graded norms $\norm_\bullet$ on $L$ such that, for each $m$, $\norm_m\in\cN^\diag(R_m)$. 

We define the \textit{strong Goldman--Iwahori distance} on $\cN(R)$ by
$$d_\infty(\norm_\bullet,\norm'_\bullet):=\sup_{m>0} m\mi d_\infty(\norm_m,\norm'_m).$$
This expression is finite by boundedness, and it coincides with the smallest constant $C>0$ such that for all $m>0$ and all $s_m\in R_m$,
\begin{equation}\label{eq:distgraded} 
e^{-Cm}\norm[s_m]'_m\leq \norm[s_m]_m\leq e^{Cm}\norm[s_m]'_m.
\end{equation}That it defines a genuine distance is clear. Endowed with the strong Goldman--Iwahori distance, $\cN(R)$ is complete by \cite[Theorem 1.3.4.1]{reb:thesis}. We insist it is not the same as the $d_\infty$-distance in \cite{bj:kstab1}, and indeed is not quite natural as it gives as much weight to behaviour in degree one as it does to asymptotic behaviour. However, it has the advantage of being a distance over a pseudodistance. We collect a few facts here that will be used consistently in this section.

\begin{proposition}\label{prop:dinftystrong}
    We have the following:
    \begin{enumerate}
        \item $\cN^\diag(R)$ is $d_\infty$-dense in $\cN(R)$;
        \item let $\norm_\bullet:m\mapsto \norm_m\in \cN(R_m)$ be a sequence of norms, and assume there exist $\norm_{\bullet,k}\in \cN(R)$ such that $\sup_m d_\infty(\norm_m,\norm_{m,k})\to_{k\to\infty} 0$. Then $\norm_\bullet\in \cN(R)$;
        \item if $d_\infty(\norm_\bullet,\norm'_\bullet)\leq c$ then $d_\infty(\norm_m,\norm'_m)\leq cm$;
        \item given a set $S\subset \cN(R)$ such that there exists $C>0$ and $\norm'_\bullet\in \cN(R)$ with, for all $\norm_\bullet\in S$, $d_\infty(\norm_\bullet,\norm'_\bullet)\leq C$, then the pointwise supremum $\norm_{\bullet,S}$ defined by
        $$\norm[s]_{m,S}:=\sup_{\norm_\bullet\in S}\norm[s]_m$$
        for $s\in R_m$ belongs to $\cN(R)$.
    \end{enumerate}
\end{proposition}
\begin{proof}
    Let $\norm_\bullet\in \cN(R)$. Pick $\varepsilon>0$, and for each $m$, consider $\norm'_{m,\varepsilon}\in \cN^\diag(R_m)$ such that $d_\infty(\norm_m,\norm'_{m,\varepsilon})<\varepsilon/3$. Pick $s_{m,i}\in R_m$ and $s_{n,j}\in R_n$. Using \eqref{eq:distgraded} we write
    \begin{align*}
        \norm[s_{m,i}\cdot s_{n,j}]'_{m+n,\varepsilon}&\leq e^{\varepsilon/3}\norm[s_{m,j}\cdot s_{n,j}]_{m+n}\\
        &\leq e^{\varepsilon/3}\norm[s_{m,i}]_m \norm[s_{n,j}]_n\\
        &\leq e^{\varepsilon}\norm[s_{m,i}]'_{m,\varepsilon}\norm[s_{n,j}]'_{n,\varepsilon}
    \end{align*}
    hence
    $$e^\varepsilon \norm[s_{m,j}\cdot s_{n,j}]'_{m+n,\varepsilon}\leq (e^{\varepsilon}\norm[s_{m,i}]'_{m,\varepsilon})(e^\varepsilon\norm[s_{n,j}]'_{n,\varepsilon}),$$
    i.e.\ the sequence of norms $m\mapsto \norm_{m,\varepsilon}:= e^{\varepsilon}\norm'_{m,\varepsilon}$ is submultiplicative. It is clear that it is bounded, and that
    $$d_\infty(\norm_\bullet,\norm_{\bullet,\varepsilon})\leq d_\infty(\norm_\bullet,\norm'_{\bullet,\varepsilon})+\varepsilon<2\varepsilon$$
    proving the density statement (1). The proof of (2) is similar: pick two sections in degrees $m$ and $n$, and given $\varepsilon>0$ fix $k$ such that $\sup_m m\mi d_\infty(\norm_m,\norm_{m,k})<\varepsilon$. Then similarly we find
    \begin{align*}
        \norm[s_{m,i}\cdot s_{n,j}]_{m+n}&\leq e^{(m+n)\varepsilon}\norm[s_{m,i}\cdot s_{n,j}]_{m+n,k}\\
        &\leq e^{(m+n)\varepsilon}\norm[s_{m,i}]_{m,k}\norm[s_{n,j}]_{n,k}\\
        &\leq e^{3(m+n)\varepsilon}\norm[s_{m,i}]_m\norm[s_{n,j}]_n.
    \end{align*}
    Thus, taking $\varepsilon\to 0$ in the inequality
    $$\norm[s_{m,i}\cdot s_{n,j}]_{m+n}\leq e^{3(m+n)\varepsilon}\norm[s_{m,i}]_m\norm[s_{n,j}]_n$$
    we find that $\norm_\bullet$ is submultiplicative. It is bounded by definition, concluding the proof of (2). (3) is definitional. (4) follows from (3) and the least upper bound property of $\cN(R)$.
\end{proof}

\subsection{Metric structures and asymptotic equivalence.}\label{sect:22}

Let $\norm_\bullet,\norm'_\bullet\in \cN(R)$. Then, the main theorem of \cite{cmac} (see also \cite[Section 4.2]{reb:1}, \cite[Theorem 9.5]{boueri}) states that the sequence of measures $m\mi_*\sigma(\norm_m,\norm'_m)$ converges weakly to a compactly supported measure on $\bbr$ that we denote by
$$\sigma(\norm_\bullet,\norm'_\bullet).$$
As in the finite-dimensional case, if $\norm_\zb\in \cN_0(R)$ and $\norm'_\zb=\norm_{\triv,\bullet}$ we will simply write $\sigma(\norm_\zb)$. It shares the same properties as its finite-dimensional counterpart Proposition \ref{prop:measure}, which we will freely use throughout this section:
\begin{proposition}\label{prop:measuregraded}
    Let $\norm_\bullet,\norm'_\bullet\in \cN(R)$, and $f:\bbr\to\bbr$.
    \begin{enumerate}
        \item if $f(\lambda)=-\lambda$, $f_*\sigma(\norm_\bullet,\norm_\bullet')=\sigma(\norm_\bullet',\norm_\bullet)$;
        \item if $f(\lambda)=\lambda+c$, then for any $a+b=c$, $f_*\sigma(\norm_\bullet,\norm_\bullet')=\sigma(e^{-\bullet b}\norm_\bullet,e^{\bullet a}\norm_\bullet')=\sigma(e^{-\bullet c}\norm_\bullet,\norm'_\bullet)=\sigma(\norm_\bullet,e^{\bullet c}\norm'_\bullet)$;
        \item if $f(\lambda)=\max(\lambda,0)$, then $f_*\sigma(\norm_\bullet,\norm_\bullet')=\sigma(\norm_\bullet,\norm_\bullet\vee\norm_\bullet')$;
        \item if $f(\lambda)=\max(\lambda,c)$, then $f_*\sigma(\norm_\bullet,\norm'_\bullet)=\sigma(\norm_\bullet,(e^{\bullet c} \norm_\bullet)\vee\norm_\bullet').$
    \end{enumerate}
\end{proposition}

We now move on to metric structures. For all $p\in [1,\infty)$, we define
$$d_p(\norm_\bullet,\norm'_\bullet)^p:=\int_\bbr |x|^pd\sigma(\norm_\bullet,\norm'_\bullet)=\lim_{m\to\infty}m\mi d_p(\norm_m,\norm'_m).$$
Clearly, for $1\leq p\leq q\leq \infty$, we have
\begin{equation}\label{eq:dpcomp}
    d_1\leq d_p \leq d_q \leq d_\infty.
\end{equation}
For $p<\infty$, $d_p$ only defines a pseudodistance, but they all define the same equivalence classes, and we write
$$\norm_\bullet\sim\norm'_\bullet\Leftrightarrow d_p(\norm_\bullet,\norm'_\bullet)=0$$
for some, hence all $p\in [1,\infty)$. Finally, we define the volume of bounded graded norms by
$$\vol(\norm_\bullet,\norm'_\bullet)=\int_\bbr xd\sigma(\norm_\bullet,\norm'_\bullet)=\lim_{m\to\infty}m\mi \vol(\norm_m,\norm'_m).$$
Again, the properties of the finite-dimensional volume from Proposition \ref{prop:vol} are easily seen to hold on $\cN(R)$:
\begin{proposition}
    Let $\norm_\bullet,\norm'_\bullet,\norm''_\bullet\in \cN(V)$, $\norm_\zb,\norm'_\zb\in \cN_0(V)$.
    \begin{enumerate}
        \item $\vol(\norm_\bullet,\norm'_\bullet)=\vol(\norm_\bullet,\norm''_\bullet)+\vol(\norm''_\bullet,\norm'_\bullet)$;
        \item $\vol(\norm_\bullet,\norm'_\bullet)=-\vol(\norm'_\bullet,\norm_\bullet)$;
        \item $\vol$ is $1$-Lipschitz with respect to $d_\infty$ in each variable;
        \item $\vol(e^{\bullet c}\norm,e^{\bullet d}\norm')=(d-c)+\vol(\norm,\norm');$
        \item $\vol$ is decreasing in the first variable and increasing in the second variable with respect to the partial order $\leq$;
        \item $\vol(t.\norm_\zb,s.\norm'_\zb)=s\cdot \vol(\norm_\zb,\norm'_\zb)+(s-t)\cdot \vol(\norm_\zb)=t\cdot \vol(\norm_\zb,\norm'_\zb)+(s-t)\vol(\norm'_\zb)$.
    \end{enumerate}
\end{proposition}

\subsection{The action of $\cN_0(R)$.}\label{sect:23}

\begin{definition}
    Let $\norm_{0,\bullet}\in \cN_0(R)$, and $\norm_\bullet\in \cN(R)$. We define $\norm_{0,\bullet}.\norm_\bullet$ to be the sequence of norms defined in each degree $m$ by $\norm_{0,m}.\norm_m$.
\end{definition}

\begin{proposition}
    The sequence of norms $\norm_{0,\bullet}.\norm_\bullet$ lies in $\cN(R)$.
\end{proposition}
\begin{proof}
    Assume first that $\norm_\bullet\in \cN^\diag(R)$. Let $s_{m,i}$ belong to a basis jointly diagonalising $\norm_m,\norm_{0,n}$, and define likewise $s_{n,j}$. Then it is clear from the envelope definition of the $\cN_0$-action that
    \begin{align*}
    \norm_{0,m+n}.\norm[s_{m,i}\cdot s_{n,j}]_{m+n}&\leq \norm[s_{m,i}\cdot s_{n,j}]_{0,m+n}\norm[s_{m,i}\cdot s_{n,j}]_{m+n}\\
    &\leq \norm[s_{m,i}]_{0,m}\norm[s_{m,i}]_m\norm[s_{n,j}]_{0,n}\norm[s_{n,j}]_n\\
    &=(\norm_{0,m}.\norm[s_{m,i}]_m )(\norm_{0,n}.\norm[s_{n,j}]_n)
    \end{align*}
    since $s_{m,i}$, $s_{n,j}$ belong to jointly diagonalising bases. Now, given $s_m=\sum_i a_i s_{m,i}$ and $s_n=\sum_j b_j s_{n,j}$ we have
    \begin{align*}
        \norm_{0,m+n}.\norm[s_m\cdot s_n]_{m+n}&\leq \max_{i,j}|a_i||b_j|\norm_{0,m+n}.\norm[s_{m,i}\cdot s_{n,j}]_{m+n}\\
        &\leq (\max_i |a_i|\norm_{0,m}.\norm[s_{m,i}]_m)(\max_j |b_j|\norm_{0,n}.\norm[s_{n,j}]_n)\\
        &=\norm_{0,m}.\norm[s_m]_m \cdot \norm_{0,n}.\norm[s_n]_n. 
    \end{align*}
    In the general case, using Proposition \ref{prop:dinftystrong}(1) we may $d_\infty$-approximate $\norm_\bullet$ by a sequence $\norm_{\bullet,k}\in \cN^\diag(R)$. By Theorem \ref{thm:dinftyenvelope} given $k$ with $d_\infty(\norm_\bullet,\norm_{\bullet,k})$ we then have for each $m$ that
    $$d_\infty(\norm_{0,m}.\norm_m,\norm_{0,m}.\norm_{m,k})\leq m\varepsilon$$
    and $\norm_{0,\bullet}.\norm_{\bullet,k}\in \cN(R)$ by the first part of the proof, so that the result follows from Proposition \ref{prop:dinftystrong}(2).
\end{proof}

As an immediate corollary of Theorems \ref{thm:dinftyenvelope}, \ref{thm:d1envelope}, and \ref{thm:voldet} we then find
\begin{corollary}\label{coro:actioncontractivity}
    Let $\norm_{0,\bullet},\norm'_{0,\bullet}\in \cN_0(R)$, $\norm_\bullet,\norm'_\bullet\in \cN(R)$. Then
    $$d_\infty(\norm_{0,\bullet}.\norm_\bullet,\norm_\zb'.\norm'_\bullet)\leq d_\infty(\norm_\zb,\norm'_\zb)+d_\infty(\norm_\bullet,\norm'_\bullet),$$
    $$d_1(\norm_{0,\bullet}.\norm_\bullet,\norm_\zb'.\norm'_\bullet)\leq d_1(\norm_\zb,\norm'_\zb)+d_1(\norm_\bullet,\norm'_\bullet),$$
    $$\vol(\norm_{0,\bullet}.\norm_\bullet,\norm_\zb'.\norm'_\bullet)=\vol(\norm_\zb,\norm'_\zb)+\vol(\norm_\bullet,\norm'_\bullet).$$
\end{corollary}

As a consequence, we easily see:
\begin{proposition}\label{prop:equivalenceaction}
    If $\norm_\zb\sim\norm_\zb'\in \cN_0(R)$ and $\norm_\bullet\sim\norm'_\bullet\in \cN(R)$, then
    $$\norm_\zb.\norm_\bullet\sim \norm'_\zb.\norm'_\bullet.$$
\end{proposition}

\subsection{Geodesics and geodesic rays.}\label{sect:24}

Given $\norm_\bullet,\norm'_\bullet\in \cN(R)$, it was seen in \cite{reb:2} that the sequence of geodesic segments $\norm_{m,t}$ joining $\norm_m$ and $\norm'_m$ is submultiplicative and bounded for each $t$. We define a \textit{geodesic} in $\cN(R)$ to be a segment of this form. 

Given that finite-dimensional geodesics are $d_p$-geodesics from Section \ref{sect:16}, it follows that such segments are indeed $d_p$-geodesics for $p\in(1,\infty)$. The CAT(0) inequality from Proposition \ref{prop:cat0} is easily seen to pass to the limit, implying that $(\cN(R)/\sim,d_2)$ is a uniquely geodesic metric space. Furthermore, $\vol$ also remains affine along geodesics in $\cN(R)$. 

We note that from the results of Section \ref{sect:16}, a $d_\infty$-limit of geodesics is a geodesic. Finally, the maximum principle Lemma \ref{lem:maxprinciple} also extends for the pointwise partial order of $\cN(R)$ to a maximum principle for geodesics in $\cN(R)$.

As in Section \ref{sect:17}, we define a \textit{geodesic ray} to be a map $[0,\infty)\ni t\mapsto \norm_{\bullet,t}\in \cN(R)$ whose restriction to each segment is a geodesic. In particular, each $t\mapsto \norm_{m,t}$ is a geodesic ray in $\cN(R_m)$, and the sequence $\ell(\{\norm_{\bullet,t}\}_t)$ lies $\cN_0(R)$, as is seen upon taking the limit in the inequality
    \begin{align*}
        \norm[s_m\cdot s_n]_{m+n,t}^\frac1t\leq \norm[s_m]_{m,t}^\frac1t \norm[s_n]_{n,t}^\frac1t.
    \end{align*}
It is also clear that $\ell$ is increasing with respect to the partial order $\leq$, and conversely that $\norm_\zb\leq \norm'_\zb$ implies the corresponding geodesic rays to be comparable as well.

By the results of Section \ref{sect:17}, geodesic rays in $\cN^\diag(R)$ are exactly those of the form $(t.\norm_\zb).\norm_\bullet$ for some $\norm_\zb\in \cN^\diag_0(R)$ and $\norm_\bullet\in \cN^\diag(R)$. Up to a spherically complete field extension, this therefore describes all possible geodesic rays in $\cN(R)$.

At this stage, we have tacitly proven Theorem A:
\begin{theorem}
    There is an action $\cN_0(R)\times \cN(R)\to \cN(R)$ which:
\begin{enumerate}
    \item is jointly $1$-Lipschitz in both variables with respect to the $d_1$ metric;
    \item transforms the scaling action on $\cN_0(R)$ into geodesic rays.
\end{enumerate}
\end{theorem}
\begin{proof}
    (2) is clear and (1) was given in Corollary \ref{coro:actioncontractivity}. 
\end{proof}

Given two geodesic rays $t\mapsto \norm_\bt,\norm'_\bt$ in $\cN(R)$, the Busemann convexity inequality derived from the finite-dimensional case Theorem \ref{thm:busemann} allows us to define
$$\hat d_p(\{\norm_\bt\}_t,\{\norm'_\bt\}_t):=\lim_{t\to\infty}\frac{d_p(\norm_\bt,\norm'_\bt)}{t}.$$
This is also a pseudodistance, but the classes induced by the $\hat d_p$ equivalence relation $\hat\tilde$ are rather large: from Theorem B, we will see that they contain all rays directed by $d_p$-equivalent norms in $\cN_0(R)$, but this distance is also independent of the starting point, so that rays starting at different point but with the same limit will be $\hat\sim$ equivalent. 

\begin{lemma}
    Let $\norm_\zb\in \cN_0(R)$ and $\norm_\bullet,\norm'_\bullet\in \cN(R)$. Then
    $$\lim_{t\to\infty}t\mi d_p((t.\norm_\zb).\norm_\bullet,(t.\norm_\zb).\norm'_\bullet)=0.$$
\end{lemma}
\begin{proof}
    By \eqref{eq:dpcomp} and Corollary \ref{coro:actioncontractivity} we have
    \begin{align*} d_p((t.\norm_\zb).\norm_\bullet,(t.\norm_\zb).\norm'_\bullet)&\leq d_\infty((t.\norm_\zb).\norm_\bullet,(t.\norm_\zb).\norm'_\bullet)\\
    &\leq d_\infty(t.\norm_\zb,t.\norm_\zb)+d_\infty(\norm_\bullet,\norm'_\bullet)\\
    &=d_\infty(\norm_\bullet,\norm'_\bullet)=o(t).
    \end{align*}
\end{proof}

\subsection{Radial isometry and spectral measures.}\label{sect:25}

We now fix a starting point $\norm_\bullet\in \cN(R)$, which by the previous results we may assume to lie in $\cN^\diag(R)$. The main result of this section is Theorem B:

\begin{theorem}\label{thm:convergencemeasure}
    Let $\norm_\zb,\norm'_\zb\in \cN_0(R)$. Then the sequence of measures
    $$\sigma_t:=t\mi_*\sigma((t.\norm_\zb).\norm_\bullet,(t.\norm'_\zb).\norm_\bullet))$$
    converges weakly to
    $$\sigma(\norm_\zb,\norm'_\zb).$$
\end{theorem}

Thus (using Proposition \ref{prop:equivalenceaction} for the first part) we obtain the following non-Archimedean analogue of the main result of \cite{finski:dp}: 
\begin{corollary}\label{coro:isometry}
    $(\cR^p(R)/\hat\sim,\hat d_p)\simeq (\cN_0(R)/\sim,d_p)$ is an isometric isomorphism; in particular, given two geodesic rays $\{\norm_{\bullet,t}\}_t,\{\norm'_{\bullet,t}\}_t\in \cR^p(R)$,
    $$\lim_{t\to\infty} t\mi d_p(\norm_{\bullet,t},\norm'_{\bullet,t})=d_p(\ell(\{\norm_{\bullet,t}\}_t),\ell(\{\norm'_{\bullet,t}\}_t)).$$
\end{corollary}

The heart of the proof of Theorem \ref{thm:convergencemeasure} lies in the following proposition, which is analogous to similar results in the Archimedean case \cite{bdl:kenergy}, \cite[Theorem 4.4.1]{reb:3}.

\begin{proposition}\label{prop:volumeasymptotics}
    Given $\norm_\zb,\norm'_\zb\in \cN_0(R)$, we have
    $$\lim_{t\to\infty}t\mi\vol((t.\norm_\zb).\norm_\bullet,(t.\norm_\zb).\norm_\bullet\vee (t.\norm'_\zb).\norm_\bullet)=\vol(\norm_\zb,\norm_\zb\vee\norm'_\zb).$$
\end{proposition}
\begin{proof}
    We proceed in three steps. We first construct a geodesic ray whose asymptotic volume captures that of the pointwise maximum of our geodesic rays. We show this geodesic ray to be minimal, which finally shows it must equal the geodesic ray on the right-hand side.

    \paragraph{1. Construction of a geodesic ray capturing volume asymptotics.}
    Fix $t\in [0,\infty)$ and consider the geodesic segment $[0,t]\ni s\mapsto \norm_{t,s,\bullet}$ joining $\norm_\bullet$ and $(t.\norm_{0,\bullet}).\norm_\bullet\vee (t.\norm'_{0,\bullet}).\norm_\bullet$. Now, since $s\mapsto (s.\norm_{0,\bullet}).\norm_\bullet$ is a geodesic starting at $\norm_\bullet$ and ending at $$(t.\norm_{0,\bullet}).\norm_\bullet\leq \norm_{t,t,_\bullet}=(t.\norm_{0,\bullet}).\norm_\bullet\vee(t.\norm_{0,\bullet}').\norm_\bullet$$ we thus have by the maximum principle Lemma \ref{lem:maxprinciple} that for all $s\in [0,t]$, $$\norm_{t,s,\bullet}\geq (s.\norm_{0,\bullet}).\norm.$$ Proceeding likewise for $\norm'_{0,\bullet}$ and taking pointwise maximum we thus have that
    \begin{align}
        \norm_{t,s,\bullet}\geq (s.\norm_{0,\bullet}).\norm_\bullet\vee(s.\norm_{0,\bullet}').\norm_\bullet.
    \end{align}
    Now fix $s$, and let $t'>t$. Applying this to $s=t<t'$, we then have that
    $$\norm_{t',t,\bullet}\geq (t.\norm_{0,\bullet}).\norm_\bullet\vee(t.\norm_{0,\bullet}').\norm_\bullet.$$
    Hence the restriction to $[0,t]$ of the geodesic $s\mapsto \norm_{t',s,\bullet}$ also starts at $\norm_\bullet$ and ends at an endpoint larger than $\norm_{t,t,\bullet}$, thus from the maximum principle again we deduce that for all $s\in [0,t]$,
    $$\norm_{t',s,\bullet}\geq \norm_{t,s,\bullet}.$$
    In other words, for fixed $s$, the sequence $t<t'\mapsto \norm_{t',s,\bullet}$ is nondecreasing. We now claim it converges in $d_\infty$ to some $\norm_{s,\bullet}$. By Proposition \ref{prop:dinftystrong}(4) it suffices to show it is $d_\infty$-bounded. Fix $t'<t$ and $s\in [0,t]$.
    \begin{align*} 
    &d_\infty(\norm_{t',s,\bullet},\norm_\bullet)\\
    &\leq \frac{s}{t'}d_\infty((t'.\norm_{0,\bullet}).\norm\vee(t'.\norm_{0,\bullet}').\norm_\bullet,\norm_\bullet)\\
    &\leq \frac{s}{t'}(d_\infty((t'.\norm_{0,\bullet}).\norm_\bullet\vee(t'.\norm_{0,\bullet}').\norm_\bullet,(t'.\norm_{0,\bullet}').\norm_\bullet)+d_\infty((t'.\norm_{0,\bullet}').\norm_\bullet,\norm_\bullet))\\
    &\leq \frac{s}{t'}(d_\infty((t'.\norm_{0,\bullet}).\norm_\bullet,(t'.\norm_{0,\bullet}').\norm_\bullet)+d_\infty((t'.\norm_{0,\bullet}').\norm_\bullet,\norm_\bullet))\\
    &\leq \frac{s}{t'}(t' d_\infty(\norm_{0,\bullet},\norm_{0,\bullet}')+t'd_\infty(\norm_{0,\bullet}',\norm_{\triv,\bullet}))\\
    &\leq s(d_\infty(\norm_{0,\bullet},\norm_{0,\bullet}')+d_\infty(\norm_{0,\bullet}',\norm_{\triv,\bullet})),
    \end{align*}
    which is finite since $\norm_{0,\bullet}'$ is bounded. We have used Busemann convexity of $d_\infty$ in the first inequality, the easily checked fact that $$d_\infty(\norm_\bullet\vee\norm'_\bullet,\norm'_\bullet)\leq d_\infty(\norm_\bullet,\norm'_\bullet)$$
    in the third inequality, and Corollary \ref{coro:actioncontractivity} together with Busemann convexity again in the fourth inequality.
    
    Being a $d_\infty$-limit of geodesics, $s\mapsto \norm_{s,\bullet}$ is itself a geodesic. This process yields a geodesic ray $[0,\infty)\ni s\mapsto \norm_{s,\bullet}$. Furthermore, $\vol$ being affine along geodesic segments, we have that for a fixed $t$,
    \begin{align*} 
        \vol(\norm_{t,s,\bullet},\norm_\bullet)&=\frac{t-s}{t}\vol(\norm_{t,0,\bullet},\norm_\bullet)+\frac{s}{t}\vol(\norm_{t,t,\bullet},\norm_\bullet)\\
        &=\frac{t-s}{t}\vol(\norm_\bullet,\norm_\bullet)+\frac{s}{t}\vol((t.\norm_{0,\bullet}).\norm_\bullet\vee (t.\norm'_{0,\bullet}).\norm_\bullet,\norm_\bullet)\\
        &=\frac{s}{t}\vol((t.\norm_{0,\bullet}).\norm_\bullet\vee (t.\norm'_{0,\bullet}).\norm_\bullet,\norm_\bullet).
    \end{align*}
    Hence, using that $\vol$ is $d_\infty$-Lipschitz,
    \begin{align*}
        \vol(\norm_{s,\bullet},\norm_\bullet)&=\lim_{t\to\infty}\vol(\norm_{t,s,\bullet},\norm_\bullet)\\
        &=s\lim_{t\to\infty}\frac{\vol((t.\norm_{0,\bullet}).\norm_\bullet\vee(t.\norm_{0,\bullet}').\norm_\bullet,\norm_\bullet)}{t}.
    \end{align*}
    Thus, the slope of $s\mapsto \vol(\norm_{s,\bullet})$ computes the limit we are interested in:
    \begin{align}\label{eq:vol}
        \lim_{t\to\infty}\frac{\vol((t.\norm_{0,\bullet}).\norm_\bullet\vee(t.\norm_{0,\bullet}').\norm_\bullet,\norm_\bullet)}{t}=\frac{d}{dt}\vol(\norm_{t,\bullet},\norm_\bullet).
    \end{align}

    \paragraph{2. Minimality of the constructed ray.}

    We now claim that $t\mapsto \norm_{t,\bullet}$ is the smallest geodesic ray bounded below by $t\mapsto (t.\norm_{0,\bullet}).\norm_\bullet$ and $t\mapsto (t.\norm'_{0,\bullet}).\norm_\bullet$. First note that it is indeed bounded below by those, by construction, as
    $$\norm_{t,t,\bullet}=(t.\norm_{0,\bullet}).\norm_\bullet\vee(t.\norm'_{0,\bullet}).\norm_\bullet.$$
    Thus, let $t\mapsto \norm'_{t,\bullet}$ be another geodesic ray satisfying this property. Then also
    $$\norm'_{t,\bullet}\geq ((t.\norm_{0,\bullet}).\norm_\bullet)\vee((t.\norm'_{0,\bullet}).\norm_\bullet)=\norm_{t,t,\bullet}$$
    for all $t$. In particular, since the restriction of $t\mapsto \norm'_{t,\bullet}$ to $[0,t]$ is a geodesic, we find by the maximum principle that for all $s\in[0,t]$,
    $$\norm'_{s,\bullet}\geq \norm_{t,s,\bullet}.$$
    Given $t'>t$, we find likewise $\norm'_{s,\bullet}\geq \norm_{t',s,\bullet}$ for all $s\in [0,t']\supset [0,t]$, thus in particular
    $$\norm'_{s,\bullet}\geq \lim_{t\to\infty}\norm_{t,s,\bullet}=\norm_{s,\bullet},$$
    proving our claim.
    
    \paragraph{3. The constructed ray is the ray directed by the maximum; conclusion.}
    
    We now show that
    \begin{align}\label{eq:finaleq}
        \norm_{t,\bullet} = (t.(\norm_{0,\bullet}\vee\norm'_{0,\bullet})).\norm_\bullet.
    \end{align}
    Note first that, since $\norm_{0,\bullet}\leq\norm_{0,\bullet}\vee\norm'_{0,\bullet}$ and $\norm'_{0,\bullet}\leq\norm_{0,\bullet}\vee\norm'_{0,\bullet}$ we have $(t.(\norm_{0,\bullet}\vee\norm'_{0,\bullet})).\norm_\bullet\geq (t.\norm_{0,\bullet}).\norm_\bullet$ and $\geq (t.\norm'_{0,\bullet}).\norm_\bullet$ for all $t$, hence
    $$(t.(\norm_{0,\bullet}\vee\norm'_{0,\bullet})).\norm_\bullet\geq (t.\norm_{0,\bullet}).\norm_\bullet\vee(t.\norm'_{0,\bullet}).\norm_\bullet.$$
    Thus $(t.(\norm_{0,\bullet}\vee\norm'_{0,\bullet})).\norm_\bullet$ is a candidate for the minimality envelope above, hence by the third part of the proof,
    \begin{equation}
        (t.(\norm_{0,\bullet}\vee\norm'_{0,\bullet})).\norm_\bullet\geq\norm_{t,\bullet}.
    \end{equation}
    On the other hand, since $\norm_{t,\bullet}\geq(t.\norm_{0,\bullet}).\norm_\bullet$ we have that $\ell(\{\norm_{t,\bullet}\})\geq \ell((t.\norm_{0,\bullet}).\norm_\bullet)=\norm_{0,\bullet}$, and likewise $\geq\norm'_{0,\bullet}$. Thus
    $$\ell(\{\norm_{t,\bullet}\})\geq \norm_{0,\bullet}\vee\norm'_{0,\bullet}.$$
    In particular,
    \begin{equation}
        \norm_{t,\bullet}=(t.\ell(\{\norm_t\}).\norm_\bullet\geq (t.(\norm_{0,\bullet}\vee\norm'_{0,\bullet})).\norm_\bullet.
    \end{equation}
    The two estimates prove \eqref{eq:finaleq}.
\end{proof}

We may now prove our main result.

\begin{proof}[Proof of Theorem \ref{thm:convergencemeasure}.]
    By a similar argument to \cite[Proposition 5.1]{cmac}, to prove Theorem \ref{thm:convergencemeasure}, it suffices to show that for all $c\in\bbr$,
    \begin{equation}\label{eq:cmac}
    \lim_{t\to\infty}\int_\bbr \max(x/t,c)\,d\sigma((t.\norm_\zb).\norm_\bullet,(t.\norm'_\zb).\norm_\bullet))\to_{t\to\infty}\int_\bbr \max(x,c)d\sigma(\norm_\zb,\norm'_\zb).
    \end{equation}
    On the one hand, using properties of the relative limit measure Proposition \ref{prop:measuregraded}, we find that
    $$\int_\bbr \max(x,c)d\sigma(\norm_\zb,\norm'_\zb)=\int_\bbr xd\sigma(\norm_\zb,(e^{c\bullet}\norm_\zb)\vee \norm'_\zb)=\vol(\norm_\zb,(e^{c\bullet}\norm_\zb)\vee \norm'_\zb).$$
    On the other hand,
    \begin{align*} 
    &\int_\bbr \max(x/t,c)\,d\sigma((t.\norm_\zb).\norm_\bullet,(t.\norm'_\zb).\norm_\bullet))\\
    &=t\mi\int_\bbr \max(x,tc)\,d\sigma((t.\norm_\zb).\norm_\bullet,(t.\norm'_\zb).\norm_\bullet))\\
    &=t\mi\int_\bbr xd\sigma((t.\norm_\zb).\norm_\bullet,(e^{tc}((t.\norm_\zb).\norm_\bullet))\vee (t.\norm'_\zb).\norm_\bullet)\\
    &=t\mi\int_\bbr xd\sigma((t.\norm_\zb).\norm_\bullet,(t.e^{tc}\norm_\zb).\norm_\bullet\vee (t.\norm'_\zb).\norm_\bullet)\\
    &=t\mi\vol((t.\norm_\zb).\norm_\bullet,(t.e^{tc}\norm_\zb).\norm_\bullet\vee (t.\norm'_\zb).\norm_\bullet).
    \end{align*}
    In the third equality, we have used used the identities $e^{\bullet c}(\norm_\zb.\norm_\bullet)=(e^{\bullet c}\norm_\zb).\norm_\bullet$ and $t.e^{\bullet c}\norm_\zb=e^{tc\bullet}(t.\norm_\zb)$, which are easily seen to hold. Applying the cocycle formula and Proposition \ref{prop:volumeasymptotics}, we then have that
    \begin{align*}
        &\lim_{t\to\infty}\int_\bbr \max(x/t,c)\,d\sigma((t.\norm_\zb).\norm_\bullet,(t.\norm'_\zb).\norm_\bullet))\\
        &=\lim_{t\to\infty}t\mi\vol((t.\norm_\zb).\norm_\bullet,(t.e^{tc}\norm_\zb).\norm_\bullet\vee (t.\norm'_\zb).\norm_\bullet)\\
        &=\lim_{t\to\infty}t\mi\vol((t.\norm_\zb).\norm_\bullet,(t.e^{tc}\norm_\zb).\norm_\bullet)\\
        &+\lim_{t\to\infty} t\mi\vol((t.e^{tc}\norm_\zb).\norm_\bullet,(t.e^{tc}\norm_\zb).\norm_\bullet\vee (t.\norm'_\zb).\norm_\bullet)\\
        &=c+\vol(e^{tc}\norm_\zb,(e^{c\bullet}\norm_\zb)\vee \norm'_\zb)\\
        &=\vol(\norm_\zb,(e^{c\bullet}\norm_\zb)\vee \norm'_\zb),
    \end{align*}
    thus proving \eqref{eq:cmac} and the theorem.
\end{proof}

\subsection{Flatness.}\label{sect:26}

Consider now $\norm_\zb\in \cN_0(R)$, and the relative limit measure $\sigma(\norm_\zb,\norm_{\triv,\bullet})=\sigma(\norm_\zb)$, whose support we denote by $S$. In \cite{reb:wn}, the author and Witt Nyström showed that for each $\norm_\zb\in \cN_0(R)$, there is an isometric embedding of the space $(\cC(S),L^p(\sigma(\norm_\zb)))$ of convex, decreasing, bounded functions on $S$ into $(\cN_0(R),d_p)$, sending $0$ to $\norm_\zb$ -- in other words, $\norm_\zb$ lies in an infinite-dimensional metric flat. In this section, we prove Theorem C building on this construction.

Consider, for each $k$, a basis $(s_{m,i})_i$ of $R_m$ that is orthogonal for $\norm_{0,m}$. Pick $f\in \cC(S)$. We define a norm $\iota[f]\norm_{0,\bullet}$ by setting
$$\iota[f]\norm[\sum a_i s_{m,i}]_{0,m}:=\max_i |a_i|_0 e^{mf(-m\mi\log\norm[s_{m,i}]_{0,m})}.$$
Analogously to \cite{reb:wn}, we have:

\begin{theorem}
    Given $f\in \cC(S)$, then $\iota[f]\norm_\zb\in \cN(R)$. Furthermore, 
    \begin{align*}
        (\cC(S),L^p(\sigma(\norm_\zb)))&\hookrightarrow (\cN_0(R),d_p),\\
        f&\mapsto (\iota[f]\norm_\zb)
    \end{align*}
    is an isometric embedding.
\end{theorem}
\begin{proof}
    We first note that, mirroring the arguments in \cite[Theorem 2.2]{reb:wn}, if $f$ has bounded right-derivative, then $\iota[f]\norm_\zb\in \cN(R)$. The isometry then follows \textit{mutatis mutandis} from the proof of \cite[Theorem 3.5]{reb:wn}.
\end{proof}

In particular, we note that:
\begin{enumerate} 
    \item for $f\equiv 0$, $\iota[f]\norm_\zb=\norm_{\triv,\bullet}$;
    \item for $f(\lambda)=-\lambda$, $\iota[f]\norm_\zb=\norm_\zb$;
    \item for $t>0$ and $f(\lambda)=-t\lambda$, $\iota[f]\norm_\zb=t.\norm_\zb$;
    \item for $f(\lambda)=-t+c$, $\iota[f]\norm_\zb=e^{\bullet c}\norm_\zb$.
\end{enumerate}

An important point in the construction is its flexibility: it only requires diagonalising $\norm_\zb$, which allows us to use it to construct flats in $\cN(R)$, much as it was used to construct flats in spaces of Kähler metrics in \cite{reb:wn}. We begin with the following:

\begin{definition}
    Let for each $m$, $s_m:=(s_{m,i})_i$ be a sequence of bases of $R_m$. We define the \textit{apartment} associated with this sequence $s_{\bullet}$ to be the set of norms $\norm_\bullet\in \cN(R)$ such that, for each $m$, $\norm_m$ admits $s_m$ as an orthogonal basis.
\end{definition}

It is clear that two norms in $\cN(R)$ lie in \textit{a} same apartment. However, unlike in the finite-dimensional case, it seems unclear how to show arbitrary apartments are nonempty, nor how to construct a "full" apartment given a norm diagonalised in a given sequence of bases. Similarly, this explains why we cannot prove Theorem \ref{thm:convergencemeasure} in the same way that we proved Theorem \ref{thm:convergencemeasurefindim}: given an apartment in which both $\norm_\zb,\norm'_\zb$ belong, we would need to construct $\norm_\bullet\in \cN(R)$ which also belongs to this apartment. However, the analogue of Remark \ref{rem:diag} still holds:

\begin{proposition}
    Let $\norm_\bullet\in \cN(R)$ and $\norm_\zb\in \cN_0(R)$. If $s_\bullet$ is a sequence of bases that is jointly orthogonal for both, then for all $\norm'_\zb$ in the apartment of $\cN_0(R)$ defined by $s_\bullet$, we have
    $$d_p(\norm_\zb.\norm,\norm'_\zb.\norm)=d_p(\norm_\zb,\norm'_\zb).$$
\end{proposition}

This gives Theorem C:
\begin{theorem}\label{thm:b}
    Let $\norm_\bullet\in \cN(R)$. Then, for all $\norm_\zb\in \cN_0(R)$, there exists an isometric embedding
    \begin{align*}
        (\cC(S),L^p(\sigma(\norm_\zb)))&\hookrightarrow (\cN(R),d_p),\\
        f&\mapsto (\iota[f]\norm_\zb).\norm_\bullet,
    \end{align*}
    and furthermore:
    \begin{enumerate}
        \item the apex $f\equiv 0$ is mapped to $\norm_\bullet$;
        \item for $t\in [0,\infty)$, setting $f_t(\lambda)=-t\lambda$, then $\iota[f_t]\norm_\zb.\norm$ is the geodesic ray $(t.\norm_\zb).\norm$.
    \end{enumerate}
\end{theorem}
\begin{proof}
    If $\norm_\bullet\in \cN^\diag(R)$, the result immediately follows from the previous proposition on picking a jointly diagonal sequence of bases for $\norm_\bullet$ and $\norm_\zb$, since the construction requires an arbitrary diagonalising basis for $\norm_\zb$.

    In the non-diagonalisable case, we $d_\infty$-approximate $\norm_\bullet$ by a sequence of norms $\norm_{\bullet,k}\in \cN^\diag(R)$, using Proposition \ref{prop:dinftystrong}(1). Then, it follows from \eqref{eq:dpcomp} and Corollary \ref{coro:actioncontractivity} that 
    $$d_p(\iota[f]\norm_\zb.\norm_{\bullet,k},\iota[f]\norm_\zb.\norm_\bullet)\leq d_\infty(\iota[f]\norm_\zb.\norm_{\bullet,k},\iota[f]\norm_\zb.\norm_\bullet)\to 0.$$
    But
    $$\norm[f-g]_{L^p(\sigma(\norm_\zb))}=d_p(\iota[f]\norm_\zb.\norm_{\bullet,k},\iota[g]\norm_\zb.\norm_{\bullet,k})$$
    for all $k$, while the right-hand side convergs to $d_p(\iota[f]\norm_\zb.\norm_{\bullet},\iota[g]\norm_\zb.\norm_{\bullet})$. This concludes the proof.
\end{proof}

\subsection{Completion and interpretation in terms of non-Archimedean pluripotential theory.}\label{sect:2na}

Unlike the case $p=\infty$, $\cN(R)/\sim$ is not \textit{a priori} complete with respect to $d_p$, and we thus define for each $p\in [1,\infty)$ the metric completion
$$\cN^p(R)$$
of $\cN(R)/\sim$ with respect to $d_p$. It is clear that, for $1\leq p \leq q$,
$$\cN^q(R)\subseteq \cN^p(R)\subseteq \cN^1(R).$$
We define likewise $\cN_0^p(R)$. Corollary \ref{coro:isometry} then extends to completions:
$$(\cN_0^p(R),d_p)\simeq (\hat \cR^p(R),d_p).$$
Under a condition known as the \textit{envelope property}, which is known in particular when $\dim X=1$ \cite{book:thu}, or when $X$ is smooth of arbitrary dimension and $\field$ is trivially valued or local of equal characteristic zero \cite{bj:ssemi}), it is shown in \cite{reb:1} that $\cN(R)/\sim$ embeds into a space of plurisubharmonic non-Archimedean metrics on the Berkovich analytification $(X\an,L\an)$ of $(X,L)$ with respect to the absolute value $|\cdot|$ on $\field$ in the sense of \cite{book:berk}. 

Following the arguments of \cite[Theorem 12.8]{bj:trivval} using the synthetic results in \cite{bj:synthetic}, one can in fact see that continuity of envelope holds if and only if $\cN_1(R)$ coincides exactly with the space $\cE^1(L)$ of \textit{finite-energy} non-Archimedean metrics on $L$. Likewise, $\cN^1_0(L)$ is exactly the space of finite-energy metrics $\cE^1_0(L)$ on the Berkovich analytification $(X_0\an,L_0\an)$ with respect to the \textit{trivial} absolute value on $\field$. In particular, elements in the completions $\cN^p(L)$ (under continuity of envelopes) and $\cN^p_0(L)$ are both realised as non-Archimedean metrics -- although, unlike the case $p=1$, there is no characterisation in terms of pluripotential theory for the elements in $\cE^1(L)\cap\cN^p(R)$. From this point of view, our results may then be interpreted in terms of the action of the space of trivially-valued non-Archimedean metrics $\cE_0^1(L)$ on $\cE^1(L)$.

\bibliographystyle{alpha}
\bibliography{bib}

@article{reb:1,
  title={The asymptotic {F}ubini-{S}tudy operator over general non-{A}rchimedean fields},
  author={Reboulet, R{\'e}mi},
  journal={Mathematische Zeitschrift},
  volume={299},
  number={3-4},
  pages={2341--2378},
  year={2021},
  publisher={Springer}
}

@article {reb:2,
    AUTHOR = {Reboulet, R\'{e}mi},
     TITLE = {Plurisubharmonic geodesics in spaces of non-{A}rchimedean
              metrics of finite energy},
   JOURNAL = {J. Reine Angew. Math.},
  FJOURNAL = {Journal f\"{u}r die Reine und Angewandte Mathematik. [Crelle's
              Journal]},
    VOLUME = {793},
      YEAR = {2022},
     PAGES = {59--103},
      ISSN = {0075-4102,1435-5345},
   MRCLASS = {32P05 (14G22 53C55)},
  MRNUMBER = {4513163},
       DOI = {10.1515/crelle-2022-0059},
       URL = {https://doi.org/10.1515/crelle-2022-0059},
}

@article{reb:3,
  title={The space of finite-energy metrics over a degeneration of complex manifolds},
  author={Reboulet, R{\'e}mi},
  journal={Journal de l’{\'E}cole polytechnique—Math{\'e}matiques},
  volume={10},
  pages={659--701},
  year={2023}
}

@article{boueri,
  title={Spaces of norms, determinant of cohomology and {F}ekete points in non-{A}rchimedean geometry},
  author={Boucksom, S{\'e}bastien and Eriksson, Dennis},
  journal={Advances in Mathematics},
  volume={378},
  pages={107501},
  year={2021},
  publisher={Elsevier}
}

@article{cmac,
  title={Distribution of logarithmic spectra of the equilibrium energy},
  author={Chen, Huayi and Maclean, Catriona},
  journal={manuscripta mathematica},
  volume={146},
  number={3-4},
  pages={365--394},
  year={2015},
  publisher={Springer}
}

@article{dar:quanti,
  title={Quantization in geometric pluripotential theory},
  author={Darvas, Tam{\'a}s and Lu, Chinh H and Rubinstein, Yanir A},
  journal={Communications on Pure and Applied Mathematics},
  volume={73},
  number={5},
  pages={1100--1138},
  year={2020},
  publisher={Wiley Online Library}
}

@book{book:berk,
  title={Spectral theory and analytic geometry over non-{A}rchimedean fields},
  author={Berkovich, Vladimir G},
  number={33},
  year={1990},
  publisher={American Mathematical Soc.}
}

@article{dar:mabuchi,
  title={The {M}abuchi geometry of finite energy classes},
  author={Darvas, Tam{\'a}s},
  journal={Advances in Mathematics},
  volume={285},
  pages={182--219},
  year={2015},
  publisher={Elsevier}
}

@article{goldmaniwahori,
 author = {Goldman, O. and Iwahori, N.},
 title = {The space of p-adic norms},
 fjournal = {Acta Mathematica},
 journal = {Acta Math.},
 issn = {0001-5962},
 volume = {109},
 pages = {137--177},
 year = {1963},
 language = {English},
 doi = {10.1007/BF02391811},
 zbMATH = {3216329},
 Zbl = {0133.29402}
}

@article{mahler,
 author = {Mahler, Kurt},
 title = {An analogue to {Minkowski}'s geometry of numbers in a field of series},
 fjournal = {Annals of Mathematics. Second Series},
 journal = {Ann. Math. (2)},
 issn = {0003-486X},
 volume = {42},
 pages = {488--522},
 year = {1941},
 language = {English},
 doi = {10.2307/1968914},
 url = {semanticscholar.org/paper/29ae2dfb8c24ca6a4a5ce4682890c591993c9401},
 zbMATH = {3042989},
 Zbl = {0027.16001}
}

@article{iwahorimatsumoto:ihes,
 author = {Iwahori, N. and Matsumoto, H.},
 title = {On some {Bruhat} decomposition and the structure of the {Hecke} rings of {{\(p\)}}-adic {Chevalley} groups},
 fjournal = {Publications Math{\'e}matiques},
 journal = {Publ. Math., Inst. Hautes {\'E}tud. Sci.},
 issn = {0073-8301},
 volume = {25},
 pages = {5--48},
 year = {1965},
 language = {English},
 doi = {10.1007/BF02684396},
 url = {https://eudml.org/doc/103854},
 zbMATH = {3362077},
 Zbl = {0228.20015}
}

@article{schneiderstuhler:inventiones,
 author = {Schneider, P. and Stuhler, U.},
 title = {The cohomology of {{\(p\)}}-adic symmetric spaces},
 fjournal = {Inventiones Mathematicae},
 journal = {Invent. Math.},
 issn = {0020-9910},
 volume = {105},
 number = {1},
 pages = {47--122},
 year = {1991},
 language = {English},
 doi = {10.1007/BF01232257},
 url = {https://eudml.org/doc/143904},
 zbMATH = {10233},
 Zbl = {0751.14016}
}

@article{drinfeld,
 author = {Drinfel'd, V. G.},
 title = {Elliptic modules},
 fjournal = {Mathematics of the USSR, Sbornik},
 journal = {Math. USSR, Sb.},
 issn = {0025-5734},
 volume = {23},
 pages = {561--592},
 year = {1976},
 language = {English},
 doi = {10.1070/SM1974v023n04ABEH001731},
 zbMATH = {3501687},
 Zbl = {0321.14014}
}

@book{fargues:progressinmath,
 author = {Fargues, Laurent and Genestier, Alain and Lafforgue, Vincent},
 title = {L'isomorphisme entre les tours de {Lubin}-{Tate} et de {Drinfeld}},
 fseries = {Progress in Mathematics},
 series = {Prog. Math.},
 issn = {0743-1643},
 volume = {262},
 isbn = {978-3-7643-8455-5},
 year = {2008},
 publisher = {Basel: Birkh{\"a}user},
 language = {French},
 zbMATH = {5152723},
 Zbl = {1136.14001}
}

@article{abbanzhuang,
 author = {Abban, Hamid and Zhuang, Ziquan},
 title = {K-stability of {Fano} varieties via admissible flags},
 fjournal = {Forum of Mathematics, Pi},
 journal = {Forum Math. Pi},
 issn = {2050-5086},
 volume = {10},
 pages = {43},
 note = {Id/No e15},
 year = {2022},
 language = {English},
 doi = {10.1017/fmp.2022.11},
 zbMATH = {7565649},
 Zbl = {1499.14066}
}

@article{dar:lines,
  title={Lines in the space of {K}ähler metrics},
  author={Darvas, Tam{\'a}s and McCleerey, Nicholas},
  journal={arXiv preprint arXiv:2507.17375},
  year={2025}
}

@incollection{bj:synthetic,
 author = {Boucksom, S{\'e}bastien and Jonsson, Mattias},
 title = {Measures of finite energy in pluripotential theory: a synthetic approach},
 booktitle = {Convex and complex: perspectives on positivity in geometry. Conference in honor of Bo Berndtsson's 70th birthday, Cetraro, Italy, October 31 -- November 2, 2022},
 isbn = {978-1-4704-7338-9; 978-1-4704-7861-2},
 pages = {159--208},
 year = {2025},
 publisher = {Providence, RI: American Mathematical Society (AMS)},
 language = {English},
 doi = {10.1090/conm/810/16207},
 zbMATH = {7973653},
 Zbl = {1556.31011}
}

@phdthesis{reb:thesis,
  title={From norms to metrics in non-Archimedean geometry},
  author={Reboulet, R{\'e}mi},
  year={2021},
  school={Universit{\'e} Grenoble Alpes [2019-2021]}
}

@article{reb:wn,
 author = {Reboulet, R{\'e}mi and Witt Nystr{\"o}m, David},
 title = {Infinite-dimensional flats in the space of positive metrics on an ample line bundle},
 fjournal = {IMRN. International Mathematics Research Notices},
 journal = {Int. Math. Res. Not.},
 issn = {1073-7928},
 volume = {2025},
 number = {22},
 pages = {17},
 note = {Id/No rnaf342},
 year = {2025},
 language = {English},
 doi = {10.1093/imrn/rnaf342},
 zbMATH = {8127406}
}

@article{bj:ssemi,
 author = {Boucksom, S{\'e}bastien and Favre, Charles and Jonsson, Mattias},
 title = {Singular semipositive metrics in non-{Archimedean} geometry},
 fjournal = {Journal of Algebraic Geometry},
 journal = {J. Algebr. Geom.},
 issn = {1056-3911},
 volume = {25},
 number = {1},
 pages = {77--139},
 year = {2016},
 language = {English},
 doi = {10.1090/jag/656},
 zbMATH = {6543159},
 Zbl = {1346.14065}
}

@inproceedings{blxz,
  title={The existence of the {K}ähler--{R}icci soliton degeneration},
  author={Blum, Harold and Liu, Yuchen and Xu, Chenyang and Zhuang, Ziquan},
  booktitle={Forum of Mathematics, Pi},
  volume={11},
  pages={e9},
  year={2023},
  organization={Cambridge University Press}
}

@article{berndt:prob,
  title={Probability measures related to geodesics in the space of {K}{\"a}hler metrics},
  author={Berndtsson, Bo},
  journal={arXiv preprint arXiv:0907.1806},
  year={2009}
}

@article{ps:geodesics,
  title={Test configurations for {K}-stability and geodesic rays},
  author={Phong, Duong H and Sturm, Jacob},
  journal={Journal of Symplectic Geometry},
  volume={5},
  number={2},
  pages={221--247},
  year={2007},
  publisher={International Press of Boston}
}

@incollection{book:gerardin,
  title={Immeubles des groupes lin{\'e}aires g{\'e}n{\'e}raux},
  author={G{\'e}rardin, Paul},
  booktitle={Non Commutative Harmonic Analysis and Lie Groups},
  pages={138--178},
  year={1981},
  publisher={Springer}
}

@Article{bj:trivval,
 Author = {Boucksom, S{\'e}bastien and Jonsson, Mattias},
 Title = {Global pluripotential theory over a trivially valued field},
 FJournal = {Annales de la Facult{\'e} des Sciences de Toulouse. Math{\'e}matiques. S{\'e}rie VI},
 Journal = {Ann. Fac. Sci. Toulouse, Math. (6)},
 ISSN = {0240-2963},
 Volume = {31},
 Number = {3},
 Pages = {647--836},
 Year = {2022},
 Language = {English},
 DOI = {10.5802/afst.1705},
 zbMATH = {7712231},
 Zbl = {1523.32036}
}

@article{berbou,
  title={Growth of balls of holomorphic sections and energy at equilibrium},
  author={Berman, Robert and Boucksom, S{\'e}bastien},
  journal={Inventiones mathematicae},
  volume={181},
  number={2},
  pages={337--394},
  year={2010},
  publisher={Springer}
}

@phdthesis{book:thu,
  title={Th{\'e}orie du potentiel sur les courbes en g{\'e}om{\'e}trie analytique non {A}rchim{\'e}dienne. {A}pplications {\`a} la th{\'e}orie d'{A}rakelov},
  author={Thuillier, Amaury},
  year={2005}
}

@article{poonen,
  title={Maximally complete fields},
  author={Poonen, Bjorn},
  journal={Enseign. Math},
  volume={39},
  number={1-2},
  pages={87--106},
  year={1993}
}

@article{bdl:kenergy,
  title={Convexity of the extended {K}-energy and the large time behavior of the weak {C}alabi flow},
  author={Berman, Robert and Darvas, Tam{\'a}s and Lu, Chinh},
  journal={Geometry \& Topology},
  volume={21},
  number={5},
  pages={2945--2988},
  year={2017},
  publisher={Mathematical Sciences Publishers}
}

@article {bbj,
    AUTHOR = {Berman, Robert J. and Boucksom, S\'{e}bastien and Jonsson,
              Mattias},
     TITLE = {A variational approach to the {Y}au-{T}ian-{D}onaldson
              conjecture},
   JOURNAL = {J. Amer. Math. Soc.},
  FJOURNAL = {Journal of the American Mathematical Society},
    VOLUME = {34},
      YEAR = {2021},
    NUMBER = {3},
     PAGES = {605--652},
      ISSN = {0894-0347,1088-6834},
   MRCLASS = {32Q20 (14E30 32P05 32Q26 32U35)},
  MRNUMBER = {4334189},
       DOI = {10.1090/jams/964},
       URL = {https://doi.org/10.1090/jams/964},
}

@article{bj:kstab1,
  title={A non-Archimedean approach to {K}-stability, {I}: {M}etric geometry of spaces of test configurations and valuations},
  author={Boucksom, S{\'e}bastien and Jonsson, Mattias},
  journal={To appear in Ann. Inst. Fourier, arXiv preprint arXiv:2107.11221},
  year={2021}
}

@article{sze:filtr,
  title={Filtrations and test-configurations},
  author={Sz{\'e}kelyhidi, G{\'a}bor},
  journal={Mathematische Annalen},
  volume={362},
  number={1},
  pages={451--484},
  year={2015},
  publisher={Springer}
}

@article{finski:metricstructure,
  title={On the metric structure of section ring},
  author={Finski, Siarhei},
  journal={arXiv preprint arXiv:2209.03853v2},
  year={2022}
}

@article{finski:sectionrings,
author = {Finski, Siarhei},
title = {Submultiplicative norms and filtrations on section rings},
journal = {Proceedings of the London Mathematical Society},
volume = {131},
number = {2},
pages = {e70077},
doi = {https://doi.org/10.1112/plms.70077},
url = {https://londmathsoc.onlinelibrary.wiley.com/doi/abs/10.1112/plms.70077},
eprint = {https://londmathsoc.onlinelibrary.wiley.com/doi/pdf/10.1112/plms.70077},
year = {2025}
}

@article{finski:dp,
 author = {Finski, Siarhei},
 title = {Geometry at infinity of the space of {K{\"a}hler} potentials and asymptotic properties of filtrations},
 fjournal = {Journal f{\"u}r die Reine und Angewandte Mathematik},
 journal = {J. Reine Angew. Math.},
 issn = {0075-4102},
 volume = {818},
 pages = {115--164},
 year = {2025},
 language = {English},
 doi = {10.1515/crelle-2024-0076},
 zbMATH = {7962753}
}

@book{bhatia,
 author = {Bhatia, Rajendra},
 title = {Positive definite matrices},
 fseries = {Princeton Series in Applied Mathematics},
 series = {Princeton Ser. Appl. Math.},
 isbn = {978-0-691-12918-1},
 year = {2007},
 publisher = {Princeton, NJ: Princeton University Press},
 language = {English},
 zbMATH = {5131267},
 Zbl = {1133.15017}
}

@incollection{bou:icm,
 author = {Boucksom, S{\'e}bastien},
 title = {Variational and non-{Archimedean} aspects of the {Yau}-{Tian}-{Donaldson} conjecture},
 booktitle = {Proceedings of the international congress of mathematicians 2018, ICM 2018, Rio de Janeiro, Brazil, August 1--9, 2018. Volume II. Invited lectures},
 isbn = {978-981-3272-91-0; 978-981-327-287-3; 978-981-3272-89-7},
 pages = {591--617},
 year = {2018},
 publisher = {Hackensack, NJ: World Scientific; Rio de Janeiro: Sociedade Brasileira de Matem{\'a}tica (SBM)},
 language = {English},
 doi = {10.1142/9789813272880_0069},
 zbMATH = {7250492},
 Zbl = {1447.32030}
}

@incollection {donaldson:symmetric,
    AUTHOR = {Donaldson, S. K.},
     TITLE = {Symmetric spaces, {K}\"{a}hler geometry and {H}amiltonian
              dynamics},
 BOOKTITLE = {Northern {C}alifornia {S}ymplectic {G}eometry {S}eminar},
    SERIES = {Amer. Math. Soc. Transl. Ser. 2},
    VOLUME = {196},
     PAGES = {13--33},
 PUBLISHER = {Amer. Math. Soc., Providence, RI},
      YEAR = {1999},
      ISBN = {0-8218-2075-3},
   MRCLASS = {58B25 (32Q20 32W20 53C55 53D20 58E11)},
  MRNUMBER = {1736211},
MRREVIEWER = {Matthew\ B.\ Stenzel},
       DOI = {10.1090/trans2/196/02},
       URL = {https://doi.org/10.1090/trans2/196/02},
}

@article {codogni:tits,
    AUTHOR = {Codogni, Giulio},
     TITLE = {Tits buildings and {$K$}-stability},
   JOURNAL = {Proc. Edinb. Math. Soc. (2)},
  FJOURNAL = {Proceedings of the Edinburgh Mathematical Society. Series II},
    VOLUME = {62},
      YEAR = {2019},
    NUMBER = {3},
     PAGES = {799--815},
      ISSN = {0013-0915,1464-3839},
   MRCLASS = {14L24 (32Q20)},
  MRNUMBER = {3974968},
MRREVIEWER = {Gergely\ B\'{e}rczi},
       DOI = {10.1017/s0013091518000512},
       URL = {https://doi.org/10.1017/s0013091518000512},
}

@article {bk:tian,
    AUTHOR = {Tian, Gang},
     TITLE = {On a set of polarized {K}\"{a}hler metrics on algebraic
              manifolds},
   JOURNAL = {J. Differential Geom.},
  FJOURNAL = {Journal of Differential Geometry},
    VOLUME = {32},
      YEAR = {1990},
    NUMBER = {1},
     PAGES = {99--130},
      ISSN = {0022-040X,1945-743X},
   MRCLASS = {32L07 (32C17 53C55)},
  MRNUMBER = {1064867},
MRREVIEWER = {John\ M.\ Lee},
       URL = {http://projecteuclid.org/euclid.jdg/1214445039},
}

@article {bk:zelditch,
    AUTHOR = {Zelditch, Steve},
     TITLE = {Szego kernels and a theorem of {T}ian},
   JOURNAL = {Internat. Math. Res. Notices},
  FJOURNAL = {International Mathematics Research Notices},
      YEAR = {1998},
    NUMBER = {6},
     PAGES = {317--331},
      ISSN = {1073-7928,1687-0247},
   MRCLASS = {32L10 (47B35)},
  MRNUMBER = {1616718},
MRREVIEWER = {Thierry\ Bouche},
       DOI = {10.1155/S107379289800021X},
       URL = {https://doi.org/10.1155/S107379289800021X},
}

@article{bj:kstab2,
  title={A non-Archimedean approach to K-stability, II: divisorial stability and openness},
  author={Boucksom, Sebastien and Jonsson, Mattias},
  journal={Journal f{\"u}r die reine und angewandte Mathematik (Crelles Journal)},
  number={0},
  year={2023},
  publisher={De Gruyter}
}

@article{kleinerleeb,
 author = {Kleiner, Bruce and Leeb, Bernhard},
 title = {Rigidity of quasi-isometries for symmetric spaces and {Euclidean} buildings},
 fjournal = {Publications Math{\'e}matiques},
 journal = {Publ. Math., Inst. Hautes {\'E}tud. Sci.},
 issn = {0073-8301},
 volume = {86},
 pages = {115--197},
 year = {1997},
 language = {English},
 doi = {10.1007/BF02698902},
 url = {https://eudml.org/doc/104123},
 zbMATH = {1171964},
 Zbl = {0910.53035}
}

@incollection{parreau:immeubles,
 author = {Parreau, Anne},
 title = {Affine buildings: construction using norms and study of isometries.},
 booktitle = {Crystallographic groups and their generalizations. II. Proceedings of the workshop, Katholieke Universiteit Leuven, Campus Kortrijk, Belgium, May 26--28, 1999},
 isbn = {0-8218-2001-X},
 pages = {263--302},
 year = {2000},
 publisher = {Providence, RI: American Mathematical Society (AMS)},
 language = {French},
 zbMATH = {1569180},
 Zbl = {1060.20027}
}

\end{document}